\documentclass[11pt,a4paper]{amsart}
\usepackage[margin=2.5cm]{geometry}
\usepackage{hyperref,enumerate,xcolor} 
\usepackage{amscd, amssymb,braket}
\usepackage{tikz-cd}
\usepackage{graphicx}
\usepackage[all]{xy}

\usepackage[nameinlink]{cleveref}
\usepackage{verbatim}

\newtheorem{thm}[equation]{Theorem}
\newtheorem{prop}[equation]{Proposition}
\newtheorem{cor}[equation]{Corollary}

\newtheorem{definition}[equation]{Definition}
\newtheorem{lemma}[equation]{Lemma}

\numberwithin{equation}{section}

\newcommand{\bbZ}{{\mathbb Z}}
\newcommand{\Z}{{\mathbb Z}}
\newcommand{\Q}{{\mathbb Q}}
\newcommand{\bbO}{{\mathbb O}}

\newcommand{\C}{{\mathbb C}}

\newcommand{\SU}{{\mathrm{SU}}}

\newcommand{\spin}{{\mathrm{spin}}}

\newcommand{\Spin}{\mathrm{Spin}}
\newcommand{\GSpin}{\mathrm{GSpin}}
\newcommand{\PGSp}{\mathrm{PGSp_6}}
\newcommand{\GSp}{\mathrm{GSp}}
\newcommand{\PGSO}{\mathrm{PGSO}}
\newcommand{\PGO}{\mathrm{PGO}}
\newcommand{\SO}{\mathrm{SO}}

\newcommand{\GSO}{\mathrm{GSO}}
\newcommand{\PD}{\mathrm{PD}}
\newcommand{\Sp}{\mathrm{Sp_6}}
\newcommand{\PGL}{\mathrm{PGL}}
\newcommand{\GL}{\mathrm{GL}}
\newcommand{\SL}{\mathrm{SL}}

\newcommand{\Aut}{\mathrm{Aut}}
\newcommand{\Hom}{\mathrm{Hom}}
\newcommand{\Ind}{\mathrm{Ind}}
\newcommand{\Irr}{\mathrm{Irr}}
\newcommand{\std}{\mathrm{std}}

\title{The Local Langlands Conjecture for $G_2$}
\author{Wee Teck Gan and Gordan Savin}
\address{W.T.G.:   Department of Mathematics, National University of Singapore, 10 Lower Kent Ridge Road
Singapore 119076} \email{matgwt@nus.edu.sg}
\address{G. S.: Department of Mathematics, University of Utah, Salt Lake City, UT 84000, USA}\email{savin@math.utah.edu}

\begin{document}

\maketitle

\begin{abstract}
 We prove the local Langlands conjecture for the exceptional group $G_2(F)$ where $F$ is a non-archimedean local
 field of characteristic zero.
   \end{abstract}
\section{\bf Introduction}  \label{S:intro}

 Let $F$ be a non-archimedean local field of characteristic $0$ and residue characteristic $p$. Let
$W_F$ be the Weil group of $F$ and let $WD_F = W_F \times \SL_2(\C)$ be the Weil-Deligne group. 
For a connected reductive group $G$ over $F$, which we assume to be split for simplicity, Langlands  conjectured that there is a surjective finite-to-one map from the set ${\rm Irr}(G)$ of (equivalence classes of) irreducible smooth representations of $G(F)$ to the set $\Phi(G)$ of (equivalence classes of)  L-parameters
\[  WD_F \longrightarrow  G^{\vee} \]
where $G^{\vee}$ is the Langlands dual group of $G$ and the homomorphisms are taken up to
$G^{\vee}$-conjugacy. This leads to a partition of the set of equivalence classes of irreducible representations of $G(F)$ into a disjoint union of finite subsets, which are the fibers of the map and are called $L$-packets.  Moreover, one would like to characterize the map 
\[ \mathcal{L}_G: \Irr(G) \longrightarrow \Phi(G) \]
 by requiring that it satisfies a number of natural conditions   and to have a refined parametrization of its fibers. 
\vskip 5pt

This local Langland conjecture (LLC)  has now been proved for $\GL_n$ by Harris-Taylor \cite{HT} and Henniart \cite{He} where the map $\mathcal{L}$ is a bijection. 
Building upon this, the LLC has now been shown
for the quasi-split classical groups (i.e. symplectic, special orthogonal and unitary groups) by the work of Arthur \cite{A}, Moeglin and Mok \cite{M}, as a consequence of the theory of twisted endoscopy using the stable twisted trace formula of $\GL_n$, and extended to pure inner forms by various authors by various means (see for example \cite{KMSW, MR, CZ1, CZ2}). It has also been shown for the group $\GSp(4)$ and its inner forms in \cite{GT} and \cite{GTW} using theta correspondence as the main tool. For general $G$, the recent work of Fargues-Scholze \cite{FS} gives a general geometric construction of a semisimplified LLC (which is a priori weaker than the full LLC).  In another direction, for general tamely ramified groups, the work of Kaletha  \cite{K1, K2} constructs an LLC for  supercuspidal L-packets. 
\vskip 5pt

\subsection{\bf The main result}
The purpose of this paper is to establish the local Langlands conjecture for the split exceptional group $G = G_2$. 
More precisely, we prove:
\vskip 10pt

\noindent{\bf \underline{Main Theorem}}
\vskip 5pt

There is a natural surjective map
\[  \mathcal{L}:  {\rm Irr} (G_2)  \longrightarrow \Phi(G_2) \]
with finite fibers satisfying the following properties:
\vskip 5pt

\begin{itemize}
 \item[(i)] $\pi \in {\rm Irr}(G_2)$ is discrete series  if and only if $\mathcal{L}(\pi)$ does not factor through a proper Levi subgropup of $G_2^{\vee}= G_2(\C)$.
 
 \vskip 5pt
 
 \item[(ii)] $\pi \in {\rm Irr}(G_2)$ is tempered  if and only if $\mathcal{L}(\pi)(W_F)$ is bounded in  $G_2^{\vee}= G_2(\C)$. More precisely, if $\pi$ is contained in an induced representation $\Ind_P^{G_2} \tau$, with $P =MN$ a parabolic subgroup and $\tau$ a unitary discrete series representation of its Levi factor $M$, then $\mathcal{L}(\pi)$ is the composite of the L-parameter of $\tau$ with the natural inclusion  $M^{\vee} \hookrightarrow G_2^{\vee}= G_2(\C)$, Here, note that the Levi subgroup $M$ is either $\GL_2$ or $\GL_1^2$, for which the LLC is known.

 \vskip 5pt

\item[(iii)]  If $\pi \in \Irr(G_2)$ is nontempered, so that $\pi$ is the unique Langlands quotient of a standard module $\Ind_P^{G_2} \tau$ induced from a proper parabolic subgroup $P = MN$, then $\mathcal{L}(\pi)$ is the composite of the L-parameter of $\tau$ with the natural inclusion $M^{\vee} \hookrightarrow G_2^{\vee}= G_2(\C)$.  \vskip 5pt

\vskip 5pt

 \item[(iv)] The map $\mathcal{L}$ features in the following two commutative diagrams:
 
 \vskip 5pt
 
 \begin{itemize}
 \item[(a)] 
\[   \begin{CD}
{\rm Irr}^{\spadesuit}(G_2)  @>\mathcal{L}>>   \Phi(G_2)  \\
@V{\theta}VV  @AA{\iota'_*}A  \\
{\rm Irr}(PD^{\times}) @>L>> \Phi(\PGL_3)    \end{CD}
 \]
 \vskip 10pt
 
 \item[(b)] 
\[  
\begin{CD}
{\rm Irr}^{\heartsuit}(G_2) @>\mathcal{L}>>   \Phi(G_2) \\
@V{\theta}VV  @VV{\iota_*}V  \\
 \Irr (\PGSp)   @.   \Phi(\PGSp) \\
@V{\rm rest}VV             @VV{\rm std}_*V  \\
 \Irr(\Sp)_{/\PGSp} @>L>> \Phi(\Sp) 
 \end{CD} 
\]


  \end{itemize}

   \vskip 5pt

 In these diagrams, the symbol $L$ refers to the known LLC for $\PD^{\times}$ and $\Sp$ and the maps on the right hand columns of the diagrams are induced by the natural maps of dual groups
 \[  
 \begin{tikzcd}
 &     &    &  \SO_7(\C) \\
  \SL_3(\C)  \arrow[r, "\iota'"]  & G_2(\C)  \arrow[r,"\iota"] &  \Spin_7(\C)  \arrow[ur, shift right=0.2ex,  "std"]  \arrow[dr, shift right=0.2ex, swap, "spin"] & \\
  &  &  &   \GL_8(\C). 
    \end{tikzcd} \]
   For the left hand columns in the two diagrams, the symbol $\theta$ refers to an appropriate theta correspondence (to be explained below) and  rest refers to the restriction of representations, with the notation $ \Irr(\Sp)_{/\PGSp}$ denoting the set of $\PGSp(F)$-orbits on $\Irr(\Sp)$.
   Moreover, $\Irr(G_2) = \Irr^{\heartsuit}(G_2) \sqcup \Irr^{\spadesuit}(G_2)$ is a decomposition of $\Irr(G_2)$ into the disjoint union of two subsets (see Theorem 
   \ref{T:intro-dichotomy} below).
   
\vskip 5pt
 \item[(v)] The map $\mathcal{L}$ is uniquely characterized by the properties (i), (ii), (iii) and (iv).

\vskip 5pt

\item[(vi)] The map $\mathcal{L}$ also features in the following commutative diagram:
 \[  \begin{CD}
  {\rm Irr}_{gen, ds}(G_2) @>\mathcal{L}>>   \Phi(G_2) \\
  @V{\theta}VV     @VV{\iota_*}V  \\
  \Irr_{gen, ds}(\PGSp)   @.   \Phi(\PGSp) \\
  @V{\spin_*}VV      @VV{\spin_*}V \\
  \Irr(\GL_8) @>L>>  \Phi(\GL_8) 
  \end{CD} \] 
Here, $\Irr_{gen, ds}(\bullet)$ refers to the subset of generic discrete series representations in $\Irr(\bullet)$, the left hand column iof the diagram is induced by the relevant maps of dual groups mentioned in (iv) above 
and $\spin_*$ refers to a certain spin lifting whose construction will be explained in \S \ref{S:triality}.     
   \vskip 5pt
   
        \vskip 5pt
     
     \item[(vii)] For each $\phi \in \Phi(G_2)$, the fiber of $\mathcal{L}$ over $\phi$ is in natural bijection with ${\rm Irr} S_{\phi}$, where
\[  S_{\phi}  = \pi_0(Z_{G_2(\C)}(\phi)),\]
when $p \ne 3$. When $p=3$, this is still true unless perhaps for $\phi \in \iota_*(\Phi_{ds}(\PGL_3))$.
Moreover, for tempered $\phi$, the trivial character of $S_{\phi}$  corresponds to the unique generic element in $\mathcal{L}^{-1}(\phi)$.

\vskip 5pt

\item[(viii)] The LLC for $G_2$ satisfies the following global-local compatibility. Suppose that  $\Pi$ is a  globally generic regular algebraic cuspidal automorphic representation of $G_2$ over a totally real number field $k$  with a Steinberg local component. Suppose that  
\[  \rho_{\Pi}: {\rm Gal}(\overline{k}/k) \longrightarrow \GL_7(\overline{\Q}_l)\]
is a Galois representation associated to  $\Pi$, in the sense that  for almost all places $v$, $\rho_{\Pi}({\rm Frob}_v)_{ss}$ is conjugate to the Satake parameter $s_{\Pi_v} \in G_2(\overline{\Q}_l)$. Then $\rho_{\Pi}$ factors through  $G_2(\overline{\Q}_l)$ (uniquely up to $G_2(\overline{\Q}_l)$-conjugation) and  for each finite place $v$, the restriction of $\rho_{\Pi}$ to the local Galois group ${\rm Gal}(\overline{k}_v/k_v)$ corresponds to the local L-parameter $\mathcal{L}(\Pi_v)$.
     \end{itemize}

  \vskip 10pt

 \vskip 5pt
 \subsection{\bf Construction of $\mathcal{L}$}
 Let us make a few remarks regarding the undefined notation in (iv) and (vi)  of the Main Theorem above and  give a brief sketch of the main ideas used in the construction of $\mathcal{L}$. The starting point of our work is the local theta correspondences furnished by the following  dual pairs:
\vskip 5pt
\[  \begin{cases} 
\PD^{\times} \times G_2 \subset E_6^{D} \\
G_2 \times \PGSp \subset E_7 \end{cases} \]
where the exceptional groups of type $E$ are  of adjoint type and $D$ denotes a cubic division $F$-algebra, so that $\PD^{\times}$ is the unique inner form of $\PGL_3$. One can thus consider the restriction of the minimal representation of $E$ to the relevant dual pair and obtain a local theta correspondence.  In particular, for a representation $\pi$ of one member of a dual pair, one has a big theta lift $\Theta(\pi)$ on the other member of the dual pair, and its maximal semisimple quotient  $\theta(\pi)$. In \cite{GS},  the following theorem was shown:


 
\vskip 5pt

\begin{thm}  \label{T:intro-dichotomy}
(i) (Howe duality) The Howe duality theorem holds for the above dual pairs, i.e. $\Theta(\pi)$ has finite length and its maximal semisimple quotient $\theta(\pi)$ is irreducible or zero. 
\vskip 5pt

(ii) (Theta dichotomy) Let $\pi \in {\rm Irr}(G_2)$. Then $\pi$ has nonzero theta lift to exactly one of $\PD^{\times}$ or ${\rm PGSp}_6$.
\vskip 5pt
\end{thm}
\vskip 5pt

In view of the above dichotomy theorem,  one has a decomposition
\[  {\rm Irr}(G_2) = {\rm Irr}^{\heartsuit}(G_2) \sqcup {\rm Irr}^{\spadesuit}(G_2)  \]
where ${\rm Irr}^{\heartsuit}(G_2)$ consists of those irreducible representations which participate in theta correspondence with $\PGSp$ and 
${\rm Irr}^{\spadesuit}(G_2)$ consists of those which participate in theta correspondence with $\PD^{\times}$.
 This explains the left hand sides of the two commutative diagrams in (iv) of the Main Theorem.
\vskip 5pt

 We can now sketch how one can define the map $\mathcal{L}:  {\rm Irr}(G_2) \longrightarrow \Phi(G_2)$. Suppose that $\pi \in {\rm Irr}(G_2)$, then we define its L-parameter $\phi_{\pi}:= \mathcal{L}(\pi)$ as follows:
\vskip 5pt

\begin{itemize}
\item if $\pi \in {\rm Irr}^{\spadesuit}(G_2)$,  say $\theta_D(\pi) = \tau_D \in \Irr(\PD^{\times})$, consider 
the Jacquet-Langlands lift $\tau$ of $\tau_D$ to $\PGL_3$ with L-parameter
$\phi_{\tau}$. Then we set
\[  \phi_{\pi} =  \iota \circ \phi_{\tau}: WD_F \longrightarrow \SL_3(\C) \subset G_2(\C). \]
where $\iota:  \SL_3(\C) \hookrightarrow G_2(\C)$ is the natural inclusion. 
\vskip 5pt

\item if $\pi \in {\rm Irr}^{\heartsuit}(G_2)$,  say $\theta(\pi)  = \sigma \in \Irr(\PGSp)$, then  
\[  {\rm rest}(\sigma)\in {\rm Irr}(\Sp(F))_{/ \PGSp(F)}, \]
 gives rise to an L-parameter (\`a la Arthur)  
\[ \phi_{{\rm rest}(\sigma)}: WD_F \longrightarrow \SO_7(\C).\] 
At this point, one needs  to show that $\phi_{{\rm rest}(\sigma)}$ factors through $G_2(\C)$, uniquely up to $G_2(\C)$-conjugacy. We shall achieve this by a global argument, using
various globalization results which are given in Appendices A and B (i.e. \S \ref{S:app} and \S \ref{S:appB}),
 the construction of Galois representations associated to cohomological cuspidal representations of $\GL_7$ and the group theoretic results of Chenevier \cite{C} and Greiss \cite{Gr95}. After this,  we set $\phi_{\pi}  :=  \phi_{{\rm rest}(\theta(\pi))}$  as a map valued in $G_2(\C)$, well-defined up to $G_2(\C)$-conjugacy. 
\end{itemize}
\noindent In other words, using the theta correspondence and  theta dichotomy, we  deduce the LLC for $G_2$ from  the known LLC for $\PD^{\times}$ and $\Sp$.  Moreover, by construction, one has the commutative diagrams in  (iv), and the various results alluded to above give the characterization of 
$\mathcal{L}$ in (v). 
\vskip 5pt

\subsection{\bf Fibers of $\mathcal{L}$}
Parametrizing the fibers of $\mathcal{L}$ (i.e. showing (vii) of the Main Theorem)  is  the most delicate  part of this paper which requires key new ideas.  In constructing the map $\mathcal{L}$, we had only needed to appeal to the LLC for $\Sp$. To explicate the fibers of $\mathcal{L}$ (for example, to show that $\mathcal{L}$ is surjective),  it  would help greatly if the LLC for $\PGSp$ is known.
  Of course, the requirement of compatibility of restriction with the LLC of ${\rm Sp}_{2n}$ places severe constraints on the LLC for ${\rm PGSp}_{2n}$, but there is an inherent ``quadratic ambiguity" on both the representation theoretic and Galois theoretic sides that the LLC for ${\rm Sp}_{2n}$ cannot resolve.
   We circumvent the lack of the full  LLC for $\PGSp$ by making crucial use of the following inputs:

      \vskip 5pt

 \begin{itemize}
 \item In his thesis and subsequent work \cite{Xu1, Xu2}, Bin Xu has resolved the quadratic ambiguities on the representation theoretic side, defining a partition of ${\rm Irr}({\rm PGSp}_{2n})$ into candidate local L-packets and establishing Arthur's multiplicity formula in the global setting based on these local partitions.  We will summarize his results that we need in \S \ref{S:Xu}; 
 
 \vskip 5pt
\item We make use of results of Kret and Shin \cite{KS} who gave a construction of Galois representations valued in $\GSpin_{2n+1}(\C)$ associated to certain cuspidal automorphic representations of  ${\rm GSp}_{2n}$. This is summarized in \S \ref{S:KS};
 \vskip 5pt
 
 \item  By a combination of the  similitude classical theta lifting from $\PGSp$ to $\PGSO_8$ and the theory of triality  on ${\rm PGSO}_8$, we construct on the representation theoretic side a Spin lifting
 \[ \spin_*: \Irr(\PGSp) \longrightarrow \Irr(\SO_8) \longrightarrow \Irr( \GL_8) \]
which should be associated to the map of dual groups
\[ \spin:  \Spin_7(\C) \longrightarrow \SO_8(\C)  \longrightarrow \GL_8(\C) \]
given by the spin representation of $\Spin_7(\C)$; this is given in \S \ref{S:triality}.  
\end{itemize}
The combination of these various inputs to understand the fibers of $\mathcal{L}$ requires an execution too delicate  to explain in the introduction. An outcome of this is the commutative diagram in (vi) of the Main Theorem, which plays a key role in the proof of (vii).  
\vskip 5pt

 This exploitation of triality to produce the Spin lifting is certainly  one of the main innovations of this paper; it will be discussed in greater detail in \S \ref{S:triality}. The idea can be pushed further to yield a weak LLC for $\PGSp$; this is carried out in Appendix C of this paper.  Such applications of triality in the global setting  were first obtained in the paper \cite{CG2} of  G. Chenevier and the first author.

   \vskip 5pt

 \vskip 10pt
 
 \subsection{\bf Prior work, related literature and further remarks}
  We conclude this introduction by  mentioning some prior work towards  the LLC for $G_2$ and their relevance to the current paper:
 \vskip 5pt
 \begin{itemize}
 \item In \cite{SW}, the second author and M. Weissman showed that an irreducible generic  supercuspidal representation of $G_2$ has nonzero theta lift to  a generic  supercuspidal  representation of exactly one of $\PGL_3$ or $\PGSp$; this is a precursor of the dichotomy theorem  of \cite{GS} mentioned above.
  \vskip 5pt

\item In \cite{HKT}, Harris-Khare-Thorne constructed a bijection
\[  {\rm Irr}_{sc, gen}(G_2(F))  \longleftrightarrow   \Phi_{sc}(G_2) \]
where the LHS refers to the set of irreducible generic supercuspidal representations, whereas the RHS refers to supercuspidal L-parameters, i.e. discrete L-parameters $W_F \longrightarrow G_2(\C)$ trivial on the Deligne $\SL_2$. In doing this, they made use of the results of Savin-Weissman \cite{SW}, Chenevier \cite{C} and Hundley-Liu \cite{HL} (on an automorphic descent from $\GL_7$ to $G_2$). In addition, they made crucial use of potential modularity results from the theory of Galois representations.  For this paper, we shall not make use of the results of \cite{HKT}. Rather, we shall reprove their result  by different means in the course of the proof of the Main Theorem. More precisely, we do not appeal to automorphic descent \cite{HL}  or potentially modularity results, but exploit the theory of triality  and \cite{SW}.
\vskip 5pt

\item In \cite{AMS}, Aubert-Moussaoui-Solleveld defined the notion  of cuspidal support for enhanced L-parameters and formulated the conjecture \cite[Conj. 7.8]{AMS} that the LLC map should be compatible with the cuspidal support maps on both sides.  In a recent preprint \cite{ATX}, Aubert-Tsai-Xu explored the implications of \cite[Conj. 7.8]{AMS} (which is formulated as \cite[Property 2.3.14]{ATX}) for the LLC of $G_2$. For $p \ne 2 $ or $3$, they took  Kaletha's construction of the LLC for   supercuspidal L-packets as a starting point 
 and  tried to use the usual desiderata of the LLC and \cite[Conj. 7.8]{AMS} as guiding principles to extend it to non-supercuspidal representations, as well as the remaining (singular) supercuspidal ones. 
 \vskip 5pt
  More precisely, for non-supercuspidal representations of $G_2$, they used essentially the same process as \cite[\S 3.5]{GS} to define the LLC and then verify that  this is consistent with \cite[Conj. 7.8]{AMS}. An interesting aspect of their work is the use of \cite[Conj. 7.8]{AMS} to limit the possible options for the enhanced L-parameters of  singular supercuspidal representations. This places nontrivial constraints on the possible enhanced L-parameters, but is ultimately not sufficient for them to obtain a definitive LLC map for $G_2$.
   \vskip 5pt
 
   From our point of view,  the association of (enhanced) L-parameters to non-supercuspidal representations of $G_2$ is  completely dictated by the usual desiderata of the LLC, as we explained in \cite[\S 3.5]{GS}. Hence, for us, the main point in constructing the LLC for $G_2$ is to take care of the supercuspidal representations, especially the non-generic ones. In particular, we do not use the results of Kaletha \cite{K1,K2} in this paper.    It will of course be interesting to see if the LLC supplied by our Main Theorem agrees with the supercuspidal L-packets constructed  by  Kaletha  \cite{K1,K2} in the case of $G_2$. 
   \vskip 5pt
   
   \item   Likewise, it will be interesting to verify the compatibility of our LLC map with the semisimple one defined by Fargues-Scholze \cite{FS}. 
   Currently, besides the case of $\GL_n$ and related groups, such compatibility results are known for $\GSp_4$ and odd unitary groups. The proofs proceed by global means, using Shimura varieties and Arthur multiplicity formula. Both these ingredients are not available in the setting of $G_2$.
   \vskip 5pt
   
   \item One may think that an unsatisfactory aspect of the characterization of $\mathcal{L}$ by the commutative diagrams in (iv) of the Main Theorem is that it is of an extrinsic nature. It would have been better to have  an intrinsic characterization that only involves invariants of representations and L-parameters of $G_2$, such as appropriate L-factors and $\epsilon$-factors. Unfortunately, there is currently no systematic local theory of L-functions for $G_2$.  On the other hand, the construction and characterization of the LLC for classical groups $G$ by twisted endoscopic transfer to $\GL_n$  in the work of Arthur \cite{A} is in the same spirit as our characterization in (vi). Namely, the LLC map $\mathcal{L}_G$ for a classical $G$ is the unique map for which one has a commutative diagram:
   \[   \begin{CD}
   \Irr(G) @>\mathcal{L}_G>> \Phi(G)  \\
   @VtVV   @VV{\rm std}_*V  \\
   \Irr(\GL_n) @>\mathcal{L}_{\GL_n}>> \Phi(\GL_n)  \end{CD} \]
   where $t$ stands for the twisted endoscopic transfer and ${\rm std}: G^{\vee} \rightarrow \GL_n(\C)$ is the standard representation.
   \vskip 5pt
   
   \item We do not address the issues of stability and endoscopic character identities for our L-packets in this paper. This is obviously a natural problem to consider. One should be able to show these using a combination of theta correspondence and the stable trace formula, along the lines of \cite{CG} for the LLC of $\GSp_4$. 
   
   \vskip 5pt

   \item In Appendix C (i.e. \S \ref{S:weakLLC}), we show how the exploitation of the principle of triality and the results of Kret-Shin allows one to construct a weak LLC map for $\PGSp$. We also explain how this refines the results of Bin Xu  \cite{Xu2} for similitude classical groups in the limited context of $\PGSp$. 
 \end{itemize}

 \vskip 10pt
 
 \section{\bf Theta Dichotomy} \label{S:dichotomy}
 We begin by recalling in greater detail  the main results of \cite{GS} that were briefly alluded to in the introduction.  
 \vskip 5pt
 
 \subsection{\bf Theta correspondences}
  In \cite{GS}, we studied the local theta correspondence for the following dual pairs:
  \[
\xymatrix@R=2pt{
&&{\rm PGSp}_6\\
&G_2 \ar@{-}[ru]\ar@{-}[ld]\ar@{-}[rd]&\\
{\rm PD^{\times}}&&{\rm PGL_3\rtimes \bbZ/2\bbZ}
}
\]
where $D$ denotes a cubic division $F$-algebra.  More precisely, one has the dual pairs
\[  \begin{cases} 
(\PGL_3 \rtimes \bbZ/2\bbZ) \times G_2 \subset E_6 \rtimes \bbZ/2\bbZ  \\ 
\PD^{\times} \times G_2 \subset E_6^{D} \\
G_2 \times \PGSp \subset E_7 \end{cases} \]
where the exceptional groups of type $E$ are of adjoint type. 
One can thus consider the restriction of the minimal representation of $E$ to the relevant dual pair and obtain a local theta correspondence.  In particular, for a representation $\pi$ of one member of a dual pair, one has a big theta lift $\Theta(\pi)$ on the other member of the dual pair, and its maximal semisimple quotient  $\theta(\pi)$. In \cite{GS},  the following theorem was shown:
\vskip 5pt

\begin{thm}  \label{T:dichotomy}
(i) (Howe duality) The Howe duality theorem holds for the above dual pairs, i.e. $\Theta(\pi)$ has finite length and its maximal semisimple quotient $\theta(\pi)$ is irreducible or zero. 
\vskip 5pt

(ii) (Theta dichotomy) Let $\pi \in {\rm Irr}(G_2)$. Then $\pi$ has nonzero theta lift to exactly one of $\PD^{\times}$ or ${\rm PGSp}_6$. In particular, one has a decomposition:
\[  {\rm Irr}(G_2) = {\rm Irr}^{\heartsuit}(G_2) \sqcup {\rm Irr}^{\spadesuit}(G_2)  \]
where ${\rm Irr}^{\heartsuit}(G_2)$ consists of those irreducible representations which participate in theta correspondence with $\PGSp$ and 
${\rm Irr}^{\spadesuit}(G_2)$ consists of those which participate in theta correspondence with $\PD^{\times}$.

\vskip 5pt

(iii)  More precisely, one has:
\vskip 5pt

\begin{itemize}
\item[(a)]  the theta correspondence for $\PD^{\times} \times G_2$ defines an injective map
 \[  \theta_{D} :  
\Irr^{\spadesuit}(G_2) \hookrightarrow 
  {\rm Irr}(\PD^{\times}), \]
 which is bijective if $p\ne 3$. Moreover, $\Irr^{\spadesuit}(G_2)$ is contained in the subset $\Irr_{ds}(G_2)$ of discrete series representations.
 \vskip 5pt
 
  \item[(b)] the theta correspondence for $G_2 \times \PGSp$ defines an injection
 \[  \theta: \Irr^{\heartsuit}(G_2)  \hookrightarrow  {\rm Irr} (\PGSp). \]
 The map $\theta$ carries tempered representations to tempered representations.
 \vskip 5pt
 
 \item[(c)]  the theta correspondence for $(\PGL_3\rtimes \bbZ/2\bbZ) \times G_2$ defines an injective map
 \[  \theta_{M_3} : \Irr^{\clubsuit}(G_2) \hookrightarrow    {\rm Irr} (\PGL_3 \rtimes \bbZ/2\bbZ), \]
 where $\Irr^{\clubsuit}(G_2) \subset \Irr^{\heartsuit}(G_2)$ is the subset of representations which participates in theta correspondence with $\PGL_3 \rtimes \bbZ/2\bbZ$.
 The map $\theta_{M_3}$ respects tempered (resp. discrete series) representations. Moreover, its image has been completely determined.

  \end{itemize}
\end{thm}
\vskip 5pt

By Theorem \ref{T:dichotomy}(iii), as a refinement of the decomposition of $\Irr(G_2)$ in Theorem \ref{T:dichotomy}(i),
one has a further decomposition
\begin{equation} \label{E:decomp}
 {\rm Irr}^{\heartsuit}(G_2) = {\rm Irr}^{\diamondsuit}(G_2) \sqcup  {\rm Irr}^{\clubsuit}(G_2), \end{equation}
where ${\rm Irr}^{\clubsuit}(G_2)$ consists of  those representations which participate in theta correspondence with $\PGL_3\rtimes \bbZ/2\bbZ$ and 
$ {\rm Irr}^{\diamondsuit}(G_2)$ consists of those which participate exclusively in the theta correspondence with $\PGSp$. Further, we have a trichotomy result for discrete series representations:
\vskip 5pt

\begin{prop}
Each irreducible discrete series representation of $G_2$ has a nonzero discrete series theta lift to exactly one of $\PD^{\times}$, $\PGL_3 \rtimes \Z/2\Z$ or $\PGSp$. 
Setting $\Irr^{\bullet}_{ds}(G_2) = \Irr_{ds}(G_2) \cap \Irr^{\bullet}(G_2)$, this trichotomy is given by the  disjoint union
\[  \Irr_{ds}(G_2) = \Irr^{\spadesuit}_{ds}(G_2) \sqcup \Irr^{\clubsuit}_{ds}(G_2) \sqcup \Irr_{ds}^{\diamondsuit}(G_2). \]
 
\end{prop}
  
 \vskip 5pt
 
 In fact, the results of \cite{GS} give a precise determination of the maps  $\theta_{M_3}$ and $\theta$ on  nonsupercuspidal (in particular nontempered) representations. These precise results are too intricate to recall here; we refer the reader to \cite{GS}.

 \vskip 10pt
 
\subsection{\bf Maps of L-parameters}
 The above dichotomy results on the representation theory side is  reflected to some extent by analogous results on the side of L-parameters.
  Recall the  natural morphisms of Langlands dual groups
\[  \begin{CD}
 \SL_3(\C) @>\iota'>> G_2(\C) @>\iota>> \Spin_7(\C) @>{\rm std}>> \SO_7(\C) \end{CD} \]
 which induce natural maps
 \[  \begin{CD} 
 \Phi(\PGL_3)   @>\iota'_*>> \Phi(G_2) @>\iota_*>> \Phi(\PGSp) @>{\rm std}_*>>\Phi(\Sp). \end{CD} \]
 Moreover, the outer automorphism of $\SL_3(\C)$ induces an action of $\Z/2\Z$ on $\Phi(\PGL_3)$. 
 We note:
 \vskip 5pt
 
 \begin{lemma} \label{L:para}
(i)  The map $\iota'_*$ gives an injection
 \[   \Phi(\PGL_3)/ _{\Z/2\Z} \hookrightarrow \Phi(G_2), \] 
 whose image we denote by $\Phi^{\spadesuit\clubsuit}(G_2)$. 
 It restricts to an injection
 \[  \Phi_{ds}(\PGL_3)/_{\Z/2\Z} \hookrightarrow \Phi_{ds}(G_2), \]
  and the image of $\Phi_{ds}(\PGL_3)$ is characterized as those $\phi \in \Phi_{ds}(G_2)$ such that the local L-factor $L(s, {\rm std} \circ \iota \circ \phi)$ has a pole at $s=0$. 
  \vskip 5pt
  
Moreover, for $\phi \in \Phi_{ds}(\PGL_3)$ with component group $S_{\phi} = \mu_3$, one has
\[  S_{\iota'_*(\phi}) = \begin{cases}
\mu_3 \text{  if $\phi$ is not self-dual;} \\
S_3 \text{   if $\phi$ is self-dual.} \end{cases} \]
 \vskip 5pt
 
 (ii) The map $\iota_*$ gives an injection
 \[  \iota_*: \Phi(G_2) \longrightarrow \Phi(\PGSp) \]
 which restricts to an injection
 \[  \Phi^{\diamondsuit}_{ds}(G_2) :=  \Phi_{ds}(G_2) \smallsetminus \Phi^{\spadesuit \clubsuit}(G_2)   \longrightarrow \Phi_{ds}(\PGSp). \]
   The image of $\Phi^{\diamondsuit}_{ds}(G_2)$ is characterized as those $\phi' \in \Phi_{ds}(\PGSp)$ such that the local L-factor $L(s,  \spin \circ \phi')$ has a pole at $s=0$. 
  Moreover, for $\phi \in \Phi^{\diamondsuit}_{ds}( G_2)$ with component group $S_{\phi}$, one has
 \[    \iota_*:  S_{\phi} \cong   S_{\iota_*(\phi)}/ Z(\Spin_7), \]
where $Z(\Spin_7) \cong \mu_2$ is the center of $\Spin_7(\C)$.
 \vskip 5pt
 
 (iii) The map 
 \[ {\rm std}_* \circ \iota_*: \Phi(G_2)  \longrightarrow \Phi(\Sp) \]
 is injective and restricts to an injection
 \[  \Phi^{\diamondsuit}_{ds}(G_2) \hookrightarrow  \Phi_{ds}(\Sp). \]
 \vskip 5pt
 
     \end{lemma}
     \begin{proof}  (i) If  $\phi_1, \phi_2  \in \Phi(\PGL_3)$ are conjugate in $G_2(\mathbb C)$, then the following semi-simple representations are isomorphic: 
     \[ 
     \phi _1+ \phi_1^{\vee} +1 \cong   \phi _2+ \phi_2^{\vee} +1. 
     \] 
     We need to show that this isomorphism implies $\phi_1\cong \phi_2$ or $\phi_1\cong \phi_2^{\vee}$.  This is obvious if $\phi_1$ is 
     irreducible. If $\phi_1$ contains an irreducible two dimensional summand, then that summand must be contained in $\phi_2$ or $\phi_2^{\vee}$. Assume that 
     it is contained in $\phi_2$. The one-dimensional summand of $\phi_1$ is the determinant inverse of the two-dimensional summand, hence it is also contained in $\phi_2$ and thus $\phi_1\cong\phi_2$.
     We leave the case of three summands as an exercise.   Assume $\phi \in \Phi_{ds}(\PGL_3)$. Then the centralizer of $\phi$ in $\SL_3(\mathbb C)$ is $\mu_3$. 
     The stabilizer in $G_2(\mathbb C)$ of $1\subset 1+  \phi + \phi^{\vee}$ is $\SL_3(\mathbb C) \rtimes \Z/2\Z$. It follows that the centralizer of $\phi$ in $G_2(\mathbb C)$ 
     is either $\mu_3$ or $S_3$, depending on whether  $\phi\not\cong \phi^{\vee}$ or $\phi\cong \phi^{\vee}$ respectively.   
  
  \vskip 5pt 
  For (ii) and (iii), we refer the reader to \cite[Chapter 4, Section 1]{GrS2}.

     \end{proof} 

\vskip 10pt

 \section{\bf Definition of $\mathcal{L}$}  \label{S:definition}
In this section, we will define the LLC  map
\[   \mathcal{L}: \Irr(G_2) \longrightarrow \Phi(G_2) \]
using the results of the previous section and establish some initial properties of $\mathcal{L}$.
\vskip 5pt

 \vskip 5pt

\subsection{\bf First definition of $\mathcal{L}$}
We shall give two slightly different constructions of the map $\mathcal{L}$. 
\vskip 5pt

For the first construction (which was sketched in the introduction), we make use of the decomposition
\[  \Irr(G_2) = \Irr^{\spadesuit}(G_2) \sqcup \Irr^{\heartsuit}(G_2) \]
introduced in Theorem \ref{T:dichotomy}(i).
First consider  $\pi \in \Irr^{\spadesuit}(G_2) \subset \Irr_{ds}(G_2)$, and set $\tau = \theta_D(\pi)$ for a unique  $\tau \in \Irr(\PD^{\times})$. By the Jacquet-Langlands correspondence and the LLC for $\PGL_3$, one has a discrete  L-parameter 
\[  \phi_{\tau} : WD_F \longrightarrow \SL_3(\C). \]
Composing this with the natural inclusion 
\[  \iota' : \SL_3(\C) \longrightarrow G_2(\C), \]
we set
\[  \mathcal{L}(\pi) = \iota' \circ \phi_{\tau}: WD_F \longrightarrow G_2(\C), \]
which is an element of $\Phi_{ds}(G_2)$ by Lemma \ref{L:para}(i).
\vskip 5pt

For $\pi \in \Irr^{\heartsuit}(G_2)$, let $\sigma = \theta(\pi) \in \Irr(\PGSp)$ and consider its restriction 
\[ {\rm rest}(\sigma) \in \Irr(\Sp)/_{\PGSp}. \]
By the LLC for $\Sp$, ${\rm rest}(\sigma)$ gives rise to an L-parameter 
\[  \phi_{{\rm rest}(\sigma)} \in \Phi(\Sp). \]
In view of Lemma \ref{L:para}, it remains to verify:
\vskip 5pt

\begin{prop}  \label{P:well}
For $\pi \in \Irr^{\heartsuit}(G_2)$, with $\sigma = \theta(\pi) \in \Irr(\PGSp)$,
\[   \phi_{{\rm rest}(\sigma)} \in {\rm Im}({\rm std}_* \circ \iota_*), \]
so that it gives rise to a uniquely determined element $\mathcal{L}(\pi) \in \Phi(G_2)$ by Lemma \ref{L:para}(iii).  
\end{prop}
We shall verify this in \S \ref{SS:well}  by global means.
\vskip 5pt

\subsection{\bf Second definition of $\mathcal{L}$} Our second construction is closely related to the first and is based on the more refined decomposition
\[  \Irr(G_2) = \Irr^{\spadesuit}(G_2) \sqcup \Irr^{\clubsuit}(G_2) \sqcup \Irr^{\diamondsuit}(G_2) \]
in (\ref{E:decomp}).
\vskip 5pt

For $\pi \in \Irr^{\spadesuit}(G_2)$ or $\Irr^{\diamondsuit}(G_2) \subset \Irr^{\heartsuit}(G_2)$, the definition of $\mathcal{L}(\pi)$ is as in the first construction above.
Thus, the difference only arises for those  $\pi \in \Irr^{\clubsuit}(G_2) \subset \Irr^{\heartsuit}(G_2)$. Such a $\pi$ has nonzero theta lift to a unique  $\tilde{\tau}  \in 
\Irr(\PGL_3 \rtimes \Z/2\Z)$. Restricting $\tilde{\tau}$ to $\PGL_3$ gives a well-defined element 
\[  \{ \tau, \tau^{\vee} \} \in \Irr(\PGL_3) / _{\Z/2\Z}, \] 
and thus an L-parameter
\[  \phi_{\tilde{\tau}}: WD_F \longrightarrow \SL_3(\C) \]
which is well-defined up to the outer automorphism action.  By Lemma \ref{L:para}, we may set
\[  \mathcal{L}(\pi) = \iota' \circ \phi_{\tilde{\tau}}  \in \Phi(G_2).\]
\vskip 5pt

There is of course a need to reconcile the two definitions for $\pi \in \Irr^{\clubsuit}(G_2)$. This follows readily from the explicit theta correspondences computed in \cite[Thm. 8.2, Thm. 8.5, Thm. 15.1 and Thm. 15.2]{GS}.
 \vskip 5pt
 
 \subsection{\bf Proof of Proposition \ref{P:well}}  \label{SS:well}
 To complete the construction of $\mathcal{L}$, it remains to verify Proposition \ref{P:well}.
 For a non-supercuspidal representation $\pi \in \Irr^{\heartsuit}(G_2)$,
  we have determined in \cite[Thm. 15.1, Thm 15.2 and  Thm. 15.3]{GS} the representation $\theta(\pi) \in \Irr(\PGSp)$ explicitly. From this, Proposition \ref{P:well} follows readily. Indeed, the results of \cite{GS} shows that $\mathcal{L}(\pi) \in \Phi(G_2)$ is given precisely by the discussion in \cite[\S 3.5]{GS}.
 
 \vskip 5pt
 It remains to treat supercuspidal $\pi \in \Irr^{\heartsuit}(G_2)$, in which case $\theta(\pi)$ is a discrete series representation.  
  We shall use a global argument. 
 \vskip 5pt
 
 Let $k$ be a totally real number field with a place $w$ such that $k_w \cong F$. Let $\mathbb{O}$ be a totally definite octonion algebra over $F$ with automorphism group $G = \Aut(\mathbb{O})$. Then $G_v$ is anisotropic at all archimedean places $v$ and $G_v$ is the split $G_2$ at all finite places $v$. 
 By Proposition \ref{P:period_global} and Corollary \ref{C:cuspidal}  in Appendix A, we can find a cuspidal automorphic representation  $\Pi = \otimes_v \Pi_v$ of $G$  so that
 \vskip 5pt
 
  \begin{itemize}
  \item  $\Pi_w \cong \pi$;
 \item  $\Pi_{\infty}$ has sufficiently regular infinitesimal character at each archiemdean place $\infty$;
  \item $\Pi_u \cong$ the Steinberg representation ${\rm St}$ at some finite place $u \ne w$. 
  \item $\Pi$ has nonzero  cuspidal global theta lift  $\Sigma := \theta(\Pi)$ to the split group $\PGSp$ over $k$.
  \end{itemize}
  Here, in the last bullet point, we are considering the dual pair $G \times \PGSp$ in the $k$-rank 3 form of $E_7$ (commonly denoted by $E_{7,3}$); this global theta correspondence was studied in \cite{GrS2}. 
 \vskip 5pt
 
 Now the nonzero cuspidal representation $\theta(\Pi) =: \Sigma = \otimes_v \Sigma_v$ of $\PGSp$ satisfies:
   \vskip 5pt
   
   \begin{itemize}
  \item $\Sigma_{\infty}$ is a holomorphic discrete series with sufficiently regular infinitesimal character for every archimedean place $\infty$. 
\item $\Sigma_w$ is the discrete series representation $\theta(\pi)$ of $\PGSp(k_w)$.  
\item $\Sigma_u$ is the Steinberg representation of $\PGSp(k_u)$. 
\end{itemize}
\vskip 5pt

By Arthur \cite{A}, the restriction of $\Sigma$ to $\Sp$ has a cuspidal transfer $\mathcal{A}({\rm rest}(\Sigma))$  to $GL_7$ (because of the Steinberg local component at $w$)  which  is regular algebraic at all real places.  Now at each finite place $v$, it follows by local-global compatibility of the transfer from classical groups to $\GL_N$ that the local L-parameter of $\mathcal{A}({\rm rest}(\Sigma))_v$ is the local L-parameter of ${\rm rest}(\Sigma_v)$. Moreover, 
by a result of Chenevier \cite[Thm. E]{C}, at all finite places $v$, the local L-parameter  of $\mathcal{A}({\rm rest}(\Sigma))_v$ factors  through $G_2(\C)$, uniquely up to $G_2$-conjugacy. Hence, 
 the local L-parameter of  ${\rm rest}(\Sigma_w) = {\rm rest}(\theta(\pi))$ takes value in $G_2(\C) \subset \SO_7(\C)$, as desired.


\vskip 5pt

We have thus completed the proof of Proposition \ref{P:well}. 

\vskip 10pt

\subsection{\bf Initial properties}
The definition of $\mathcal{L}$ given above and the results of \cite{GS} allows us to read off some initial properties of  $\mathcal{L}$ readily. We highlight some of these here:
\vskip 5pt

 \begin{itemize}
 \item (Nontempered representations) As mentioned above, the results of \cite[Thm. 15.1]{GS} shows that our definition of $\mathcal{L}$ gives the Langlands parametrization of nontempered representations discussed in \cite[\S 3.5]{GS} (based on the consideration of usual desiderata). In particular, $\mathcal{L}$ gives  a bijection
 \[  \mathcal{L}: \Irr_{nt}(G_2) \longleftrightarrow \Phi_{nt}(G_2) \]
 where the subscript ``nt" stands for ``nontempered". In particular, nontempered L-packets of $G_2$ are singletons. 
 
 \vskip 5pt
 
 \item (Tempered representations) The results of \cite[Thm 15.2, 15.3]{GS} imply that one has a commutative diagram
 \[ \begin{CD}
   \Irr_t(G_2)@>\mathcal{L}>>\Phi_t(G_2)  \\
@AAA   @AAA \\
\Irr_{ds}(G_2) @>\mathcal{L}>> \Phi_{ds}(G_2). \end{CD} \]
\vskip 5pt

\item (Tempered but non-discrete series) 
One has a surjection
\[  \mathcal{L}: \Irr_t(G_2) \smallsetminus \Irr_{ds}(G_2) \twoheadrightarrow \Phi_t(G_2) \smallsetminus \Phi_{ds}(G_2)  \]
such that the fiber  over $\phi$ has size $1$ or $2$, according to whether $|S_{\phi}| =1$ or $2$. The fiber $\mathcal{L}^{-1}(\phi)$ has  a unique  generic representation, and we attach it to the trivial character of $S_{\phi}$.   
  
 \vskip 5pt

 \item (Summary) In summary, we have established items (i), (ii) and (iii) of the Main Theorem. Moreover, the commutativity of the two diagrams   in (iv)  of the Main Theorem follows by the construction of $\mathcal{L}$, as does the characterization of $\mathcal{L}$ given in (v) of the Main Theorem. Furthermore, we have also demonstrated (vii) of the Main Theorem for non-discrete-series parameters $\phi$ (for all $p$).
 \vskip 5pt
 
 \item (Compatibility with global Langlands) We can also show (viii) of the Main Theorem.  In the context of (viii) in the Main Theorem, the global theta lift $\Sigma$ of $\Pi$ to $\PGSp$ is globally generic, regular algebraic and  cuspidal with a Steinberg local component. Moreover, the transfer of ${\rm rest}(\Sigma)$ to $\GL_7$ is a regular algebraic cuspidal automorphic representation  $\mathcal{A}({\rm rest}(\Sigma))$  (with cuspidality a consequence of the Steinberg local component).
 Hence, the Galois representation    $\rho_{\Pi}$ is associated with $\mathcal{A}({\rm rest}(\Sigma))$.  By \cite[Thm. 6.4]{C}, $\rho_{\Pi}$ takes value in $G_2(\overline{\Q}_l)$ (after conjugation). By local-global compatibility \cite{TY, Ca}, the local Galois representation $\rho_{\Pi, v}$ at each finite place $v$ corresponds to the local L-parameter of ${\rm rest}(\Sigma_v)$. But by the construction of  $\mathcal{L}$, the local L-parameter of ${\rm rest}(\Sigma_v)$ is ${\rm std} \circ \iota \circ  \mathcal{L}(\Pi_v)$. This establishes (viii). 
  \end{itemize}
 
 \vskip 10pt
 
 \subsection{\bf The packets for $\Phi_{ds}^{\spadesuit \clubsuit}(G_2)$}
 By the above discussion,  it remains to understand the map
 \[  \mathcal{L}: \Irr_{ds}(G_2) \longrightarrow \Phi_{ds}(G_2) \]
 on discrete series representations.  In fact, by construction, $\mathcal{L}$ restricts to give:
 \[  \mathcal{L}^{\spadesuit \clubsuit}: \Irr_{ds}^{\spadesuit}(G_2) \sqcup \Irr_{ds}^{\clubsuit}(G_2) \longrightarrow \Phi_{ds}^{\spadesuit \clubsuit}(G_2). \]
 and
 \[  \mathcal{L}^{\diamondsuit}: \Irr_{ds}^{\diamondsuit}(G_2) \longrightarrow \Phi_{ds}^{\diamondsuit}(G_2). \]
 The results of \cite[\S 7 and \S 8]{GS} allows us to understand the fibers of the map $\mathcal{L}^{\spadesuit \clubsuit}$ quite precisely.
 \vskip 5pt
 
 \begin{prop}
  The map $\mathcal{L}^{\spadesuit\clubsuit}$ is surjective.  Moreover, for $\phi \in \Phi^{\spadesuit \clubsuit}_{ds}(G_2)$, 
  the fiber $\mathcal{L}^{-1}(\phi)$ contains a unique generic element (belonging to $\Irr_{ds}^{\clubsuit}(G_2)$). Further,
  for $p \ne 3$, one has a natural bijection
  \[  \mathcal{L}^{-1}(\phi) \longleftrightarrow \Irr (S_{\phi})  \]
 such that the unique generic element corresponds to the trivial character of $S_{\phi}$. 
 For $p= 3$, we only know that  there is an injection $ \mathcal{L}^{-1}(\phi) \hookrightarrow \Irr (S_{\phi})$ in general, though this injection is bijective when $\phi|_{\SL_2}$ is the subregular $\SL_2$ in $G_2(\C)$.
 \end{prop}
 \vskip 5pt
 
 \begin{proof}
 As almost all statements of the proposition follow from \cite[\S 7 and \S 8]{GS}, we will focus on the main new information here: the natural bijection
  \[  \mathcal{L}^{-1}(\phi) \longleftrightarrow \Irr (S_{\phi})  \]
 when $p \ne 3$.  To obtain this, we need to set things up rather carefully. 
 \vskip 5pt
 
 Let $M_3$ be the $F$-algebra of $3 \times 3$-matrices, with associated Jordan algebra $M_3^+$. The automorphism group $\Aut(M_3^+)$  of $M_3^+$ is a disconnected group whose  identity component is $H = \PGL_3$, with Langlands dual group $H^{\vee}= \SL_3(\C)$. Note that $H^1(F, H) = H^1(F, \PGL_3)$ classifies the isomorphism classes of central simple $F$-algebras of degree $3$, with the distinguished point corresponding to $M_3$. Moreover,
  the invariant map gives a bijection:
 \[ {\rm inv}:  H^1(F, H) = H^1(F, \PGL_3) \longrightarrow \Z/3\Z.  \] 
  For each nontrivial $j \in \Z/3\Z$, the associated central simple algebra $D_j$ is a cubic division algebra. Moreover, $D_1$ and $D_2$ are opposite of each other, so that the inverse map $x \mapsto x^{-1}$ gives a canonical  isomorphism of multiplicative groups $i: D_1^{\times} \cong D_2^{\times}$.
  \vskip 5pt
  
  Suppose that $\tau$ is a discrete series representation of $\PGL_3$. Then for $j \in \Z/3\Z$, one has a Jacquet-Langlands transfer $JL_{D_j}(\tau)$ of $\tau$ to $D_j^{\times}$. 
  Via the isomorphism $i: D_1^{\times} \cong D_2^{\times}$, we may compare $JL_{D_1}(\tau)$ and $JL_{D_2}(\tau)$. One has
  \[  JL_{D_1}(\tau)^{\vee} \cong i^*(JL_{D_2}(\tau)). \]
  
  \vskip 5pt

  Consider now  $[\phi] \in \Phi_{ds}^{\spadesuit \clubsuit}(G_2)$, which is a $G_2$-conjugacy class of homomorphisms $\phi: WD_F \longrightarrow G_2(\C)$.
 One has a projective system (relative to $G_2$-conjugacy) of centralizer subgroups $Z_{G_2}(\phi)$ which are finite of order $3$, and thus identified with the component groups $S_{\phi}$. We denote this projective system of local component groups by $S_{[\phi]}$ and  
 set $C_{[\phi]} = Z_{G_2}(S_{[\phi]})$: this is a projective system of subgroups of $G_2(\C)$ which are simply-connected of type $A_2$. We fix a projective system (relative to $G_2$-conjugacy) of isomorphisms
 \[  i_{[\phi]} : H^{\vee} = \SL_3(\C)  \longrightarrow C_{[\phi]}   \]
 Note that there are basically two inner automorphism classes of such projective systems of isomorphisms, exchanged by an outer automorphism of $\SL_3(\C)$.
 \vskip 5pt

 By virtue of $i_{[\phi]}$,  one has:
 \vskip 5pt
 
 \begin{itemize}
 \item[(a)]  $[\phi]$ gives rise to a well-defined $\SL_3(\C)$-conjugacy classes of maps 
 \[  [\rho]: WD_F \longrightarrow \SL_3(\C), \]
 i.e. a discrete series L-parameter of $H = \PGL_3$, so that
 \[  [\phi] = i_{[\phi]} \circ [\rho].  \]  
 Let $\tau_{\rho}$ be the  discrete series representation of $\PGL_3$ with L-parameter $\rho$.
 \vskip 5pt
 
 \item[(b)]   One also has a projective system of  isomorphism
 \[  
i_{[\phi]}  : Z(H^{\vee})  = Z(\SL_3(\C)) \longrightarrow  S_{[\phi]} = Z(C_{[\phi]})     \]
inducing a projection system of bijections:
\[  \begin{CD}
j: \Irr(S_{[\phi]}) @>{i_{[\phi]}^*}>>  \Irr(Z(H^{\vee})) @>{\rm Kottwitz}>> H^1(F, \PGL_3) @>{\rm inv}>> \Z/3\Z. \end{CD} \]
Hence, an element $\eta \in \Irr(S_{[\phi]})$ gives rise to   a central simple algebra 
$D_{j(\eta)}$ of degree $3$, determined by its invariant $j(\eta) \in \Z/3\Z$.
 \vskip 5pt
 
 \item[(c)]  By combining (a) and (b), one sees that $\eta \in \Irr(S_{[\phi]})$ gives rise to a representation
 \[  \tau_{\phi, \eta} = JL_{D_{j(\eta)}}(\tau_{\rho}) \in \Irr(PD_{j(\eta)}^{\times}). \]
 \end{itemize}
 \vskip 5pt
 
 Now the central simple algebra $D_{j(\eta)}$, or rather its associated Jordan algebra $D_{j(\eta)}^+$ gives rise to a local theta correspondence for the dual pair
 \[  \Aut(D_{j(\eta)}^+)  \times G_2 \subset E_6^{D_{j(\eta)}^+}. \]
 This is the local theta correspondence used in the definition of $\theta_{D}$ and $\theta_{M_3}$ in the Main Theorem. Observe that $D_{j(\eta)}^+ \cong D_{-j(\eta)}^+$ and we have  canonical (up to inner automorphisms) isomorphisms 
 \[ 
 PD_{j(\eta)}^{\times} =  \Aut(D_{j(\eta)}^+)^0 \cong \Aut(D_{-j(\eta)}^+)^0 = PD_{-j(\eta)}^{\times}, 
 \] 
 where  the composite $i : PD_{j(\eta)}^{\times} \rightarrow PD_{-j(\eta)}^{\times}$ is given by $i(x)=x^{-1}$ (up to inner automorphisms).  
  Hence one can define a representation 
 \[ \pi_{\eta} := \theta( \tau_{\phi, \eta}) = \theta( JL_{D_{j(\eta)}(\tau_{\rho}}) ) \cong \theta( JL_{D_{-j(\eta)}}(\tau_{\rho}^{\vee}) ),  \]  
 where for the last isomorphism we used the fact  that $i^* ( JL_{D_{j(\eta)}}(\tau_{\rho}))= JL_{D_{-j(\eta)}}(\tau_{\rho}^{\vee})$ as observed earlier.  
 When $p \ne 3$, we have seen in \cite{GS} that $\pi_{\eta}$ is nonzero irreducible. When $p =3$, we only know this nonvanishing when $\eta$ is trivial  or if $\phi$ gives the subregular $\SL_2$ when restricted to the Deligne $\SL_2$. 
 \vskip 5pt
 
 It is clear from the construction of $\mathcal{L}$ that  one has
 \[  \mathcal{L}^{-1}([\phi]) = \{  \pi_{\eta}: \eta \in  \Irr (S_{[\phi]})  \} \quad \text{for each $[\phi] \in \Phi^{\spadesuit\clubsuit}_{ds}(G_2)$.} \]
 This gives the natural parametrization of the fibers of $\mathcal{L}^{\spadesuit\clubsuit}$ in the proposition. 
 \vskip 5pt
 
 Note that this parametrization does not depend on the choice of the projective system of isomorphism $i_{[\phi]}$. Indeed, there are 2 such choices as noted above, but changing the choice 
replaces $\rho$ by $\rho^{\vee}$  in (a) above  (and hence $\tau_{\rho}$ by $\tau_{\rho}^{\vee}$)  and $j$ by $-j$ in (b). Hence, under this new regime, $\eta \in \Irr(S_{[\phi]})$ is associated by the same process to 
\[  \theta(JL_{D_{-j(\eta)}}(\tau_{\rho}^{\vee}))  \cong  \theta( JL_{D_{j(\eta)}}(\tau_{\rho}) ) = \pi_{\eta}, \]
 as asserted. 
  \end{proof}

 \vskip 10pt
 
 After the above proposition, we see that to prove the Main Theorem, it remains to analyze the map
 \[  \mathcal{L}^{\diamondsuit}: \Irr_{ds}^{\diamondsuit}(G_2) \longrightarrow \Phi_{ds}^{\diamondsuit}(G_2). \]
  For example, a first key question to address is whether this map is surjective. Before addressing such questions, let us describe in the following two sections some ingredients we shall use.
 \vskip 10pt

 \section{\bf Triality and Spin Lifting}  \label{S:triality}

In this section,  we explain how the principle of triality, together with the theory of theta correspondence, can be used to construct a candidate Spin functorial lifting
\[  \spin_* : \Irr_{gen}(\PGSp)  \longrightarrow \Irr_{gen}({\rm PGSO}_8) \longrightarrow \Irr(\GL_8). \]
This also explains the  diagram in (vi) of the Main Theorem.

\vskip 5pt

\subsection{\bf Triality and Spin representations.}  \label{SS:triality}
We begin with a brief discussion of the phenomenon of triality. 
Recall that there are 3 different maps over $F$:
\[  f_1, f_2, f_3: \SO_8 \longrightarrow \PGSO_8 \]
where the groups $\SO_8$ and $\PGSO_8$ are split.
  These three morphisms are non-conjugate under $\PGSO_8$ but are cyclically permuted by an order 3 outer automorphism of $\PGSO_8$ (the triality automorphism).  
Moreover, they induce corresponding morphisms 
\[  f_1^{\vee}, f_2^{\vee}, f_3^{\vee} : \Spin_8(\C) \longrightarrow \SO_8(\C) \]
on the dual side.
Without loss of generality, let us fix
\[  f_1 : \SO_8  \hookrightarrow \GSO_8  \longrightarrow \PGSO_8 \]
and call it the ``standard $\SO_8$''. Then $f_1^{\vee} = \std:  \Spin_8(\C) \longrightarrow \SO_8(\C)$ will be called the standard representation of $\Spin_8(\C)$. On the other hand, $f_2^{\vee}$ and $f_3^{\vee}$ will be called the half-spin representations of $\Spin_8(\C)$.
\vskip 5pt

Consider now the standard embedding 
\[    \begin{CD}
\Spin_7(\C)  @>i>>  \Spin_8(\C)  \\
@VVV  @VVf_1^{\vee}V  \\
  \SO_7(\C) @>>>\SO_8(\C). \end{CD} \]
Then as representations of $\Spin_7(\C)$, one has:
\[  f_1^{\vee} \circ i   = \std \oplus 1 \quad \text{and} \quad f_2^{\vee} \circ i = f_3^{\vee} \circ i = \spin. \]
 Thus, one has a very useful and convenient description of the Spin representation of $\Spin_7(\C)$.

\vskip 10pt

\subsection{\bf Implication.} We observe some implications in representation theory.
\vskip 5pt

For an irreducible generic representation $\sigma$ of $\PGSp$, and $\sigma^{\flat} \in \Irr(\Sp)$ contained in the restriction ${\rm rest}(\sigma)$, 
one may consider the usual (isometry) theta lift of $\sigma^{\flat}$ from $\Sp$ to $\SO_8$, as well as the similitude theta lift of $\sigma$ to $\PGSO_8$. Both these theta lifts are nonzero and  one has the following compatibility: 
\[  \theta(\sigma^{\flat})   \subset \theta(\sigma)|_{\SO_8}  = f_1^*(\theta(\sigma)). \] 
If the L-parameter of $\sigma^{\flat}$ is $\phi^{\flat}$, then the L-parameter of any irreducible constituent of  $f_1^*(\theta(\pi))$ is $\phi^{\flat} \oplus 1$.
\vskip 5pt

On the other hand, if one considers $f_2^*(\theta(\sigma))$ instead, one potentially gets a very different representation of $\SO_8$ with a very different L-parameter. Indeed, consider the special case when $\pi$ is unramified with Satake parameter $s \in \Spin_7(\C)$. Then one can show that $\theta(\sigma)$ is unramified with Satake parameter $i(s) \in \Spin_8(\C)$. Thus, the Satake parameter for $f_1^*(\theta(\sigma)$ is $f_1^{\vee}(i(s))  = 1 \oplus s  \in \SO_8(\C)$ (as we noted above), whereas that of $f_2^*(\theta(\pi))$ is
\[  f_2^{\vee}(i(s))  =  \spin (s). \]
This suggests that  the map on generic representations of $\PGSp$ given by
\[  \sigma\mapsto \theta(\sigma) \mapsto f_2^*(\theta(\sigma))  \in \Irr(SO_8)/ _{\PGSO_8}  \]
is nothing but the Langlands functorial lifting corresponding to the spin representation 
\[  \spin: \Spin_7(\C) \longrightarrow \SO_8(\C) \longrightarrow \GL_8(\C), \]
after one composes the above with the functorial lifting 
\[ \mathcal{A}:  \Irr(SO_8) \longrightarrow \Irr(\GL_8)\]
 provided by Cogdell-Kim-Piatetski-Shapiro-Shahidi \cite{CKPSS}) or Arthur \cite{A}.

\vskip 10pt

\subsection{\bf Spin lifting.}   \label{SS:spin-lift}
 Motivated by the above, we can now define the Spin lifting 
 \[  \spin_*:  \Irr_{gen}(\PGSp)  \longrightarrow \Irr(\GL_8)\]
   by
   \[  \spin_*(\sigma) = \mathcal{A} \left(  f_2^*(\theta(\sigma))  \right) \in \Irr(\GL_8). \]
\vskip 5pt

Observe that the above construction can be carried out globally as well. Namely, if $\Sigma$ is a globally generic cuspidal automorphic representation of $\PGSp$ over a number field $k$, then the global theta lift $\theta(\Sigma)$  on $\PGSO_8$ is globally generic and thus nonzero. Suppose further that $\theta(\Sigma)$ is cuspidal. Then $f_2^*(\theta(\Sigma))$ gives rise to a submodule of globally generic cusp forms on $\SO_8$, all of whose summands are weak spin lifting of $\Sigma$. Its transfer
$\mathcal{A} \left(  f_2^*(\theta(\Sigma))  \right)$
 to $\GL_8$  (\`a la Cogdell-Kim-Piatetski-Shapiro-Shahidi \cite{CKPSS}) is then an isobaric automorphic representation of $\GL_8$ (of orthogonal type): this is  the global analog of the local construction above. Moreover, for each place $v$, one has:
 \[  \spin_*(\Sigma)_v \cong \spin_*(\Sigma_v) \in \Irr(\GL_8(k_v)). \]
 
 \vskip 10pt
 
 \subsection{\bf A key property}
 We shall now show a key property of our local spin lifting $\spin_*$.

 \vskip 5pt
 
\begin{prop}  \label{P:key}
Let $\sigma$ be an irreducible generic  discrete series representation of $\PGSp(F)$ and set   $\tau = \spin_*(\sigma) \in \Irr(\GL_8)$.   
Then  the following holds:
\begin{itemize} 
\item[(i)]   For the $\gamma$-factors provided by the Langlands-Shahidi theory, 
\[ 
\gamma(s, \sigma, \mathrm{spin}, \psi) = \gamma(s, \spin_*(\sigma), \mathrm{std},\psi),
\] 
for any non-trivial character $\psi$ of $F$. 
\vskip 5pt

\item[(ii)]  Let $\chi$ be a quadratic character of $F^{\times}$.  Then 
\[  \spin_*(\sigma \otimes \chi) = \spin_*(\sigma) \otimes \chi. \] 
 \end{itemize}  
\end{prop} 
\begin{proof} 
We shall prove this via a classical global to local argument. 
Choose a totally complex number field $k$ with some finite set $T \sqcup \{w, w'\}$ of finite places (with $T$ nonempty) such that $k_u \cong k_w \cong k_{w'} \cong F$ for all $u \in T$. 
By the globalization result in Proposition \ref{P:period_global_2} of Appendix A,  there exists a globally generic cuspidal automorphic representation $\Sigma$ such that 
\begin{itemize} 
\item $\Sigma_u\cong \sigma$ for all $u \in T$; 
\item $\Sigma_w$ is the Steinberg representation ${\rm St}$;
\item $\Sigma_{w'}$ is a generic supercuspidal representation $\sigma'$;
\item $\Sigma_v$ is a principal series representation (induced from the Borel subgroup) at all other finite places $v$.   
\end{itemize}  
 \vskip 5pt
 
 Fix a global additive character $\Psi$ of $k \backslash \mathbb{A}_k$ with $\Psi_w =\Psi_{w'} =  \Psi_u = \psi$ for all $u \in T$. For a sufficiently large set  $S$ of places,  containing $w$, $w'$, $T$ and all archimedean places, we have a global functional equation 
\[ 
L^S(1-s, \Sigma^{\vee}, \mathrm{spin}) = \prod_{v\in S} \gamma(s, \Sigma_v, \mathrm{spin}, \Psi_v)\cdot  L^S(s,  \Sigma ,\mathrm{spin}) 
\] 
where $L^S$ denotes the  partial $L$-function. Since $\Sigma$ is generic, its theta lift $\theta(\Sigma)$ to $\mathrm{PGSO}_8$  is non-zero and generic. Moreover, 
since $\Sigma_w$ is the Steinberg  representation, its theta lift to  $\mathrm{PGSO}_6$ (the smaller group in the Rallis-Witt tower) is 0. Hence  $\theta(\Sigma)$ is cuspidal.  
 \vskip 5pt
 
 By the global analog of the spin lifting discussed in the previous subsection, we have an isobaric generic automorphic representation
 \[  \spin_*(\Sigma) = \mathcal{A} \left(  f_2^*(\theta(\Sigma))  \right) \]
 on $\GL_8$.  Likewise, we have the global functional equation 
\[ 
L^S(1-s, \spin_*(\Sigma)^{\vee}, \mathrm{std}) = \prod_{v\in S} \gamma(s, \spin_*(\Sigma)_v, \mathrm{std}, \Psi_v)\cdot  L^S(s, \spin_*( \Sigma), \mathrm{std}) 
\] 
\vskip 5pt

Next, by the theta correspondence of unramified representations, observe that  
\[  L^S(s, \Sigma, \mathrm{spin}) =L^S(s, \spin_*(\Sigma), \mathrm{std}). \]
Moreover, one can readily check that
\[  \gamma(s, \Sigma_v, \mathrm{spin}, \Psi_v) = \gamma(s, \spin_*(\Sigma)_v, \mathrm{std}, \Psi_v) \quad \text{for all $v \notin T \sqcup \{w , w'\}$.} \]
   Indeed, the local components at finite places outside of $T \sqcup \{w, w' \}$  are generic principal series representations and hence their spin lifts and $\gamma$-factors are easily computable. 
     For complex groups, 
 the theta correspondence  is explicitly known, \cite{AB} (observe here that $\PGSp(\mathbb C)=\Sp(\mathbb C)/\mu_2$ so the usual theta correspondence is already a correspondence for similitude groups), so that one can check the identities at complex places. Hence, we deduce that
 \[   \gamma(s, \Sigma_w, \mathrm{spin}, \Psi_w)  \cdot \gamma(s, \Sigma_{w'}, \mathrm{spin}, \Psi_{w'})  \cdot \prod_{u \in T} \gamma(s, \Sigma_v, \mathrm{spin}, \Psi_u)   \]
 \[=
 \gamma(s, \spin_*(\Sigma)_w, \mathrm{std}, \Psi_w)  \cdot  \gamma(s, \spin_*(\Sigma)_{w'}, \mathrm{std}, \Psi_{w'})  \cdot \prod_{u \in T} \gamma(s, \spin_*(\Sigma)_u, \mathrm{std}, \Psi_u). \] 
 This gives the identity for representations of $G_2(F)$:
 \[     \gamma(s, {\rm St}, \mathrm{spin}, \psi) \cdot  \gamma(s, \sigma', \mathrm{spin}, \psi)   \cdot  \gamma(s, \sigma, \mathrm{spin}, \psi)^{|T|}  \]
 \[ = 
 \gamma(s, \spin_*({\rm St}), \mathrm{std}, \psi)  \cdot  \gamma(s, \spin_*(\sigma'), \mathrm{std}, \psi) \cdot \gamma(s, \spin_*(\sigma), \mathrm{std}, \psi)^{|T|}. \] 
  Since we can vary the size of $T$, by considering the quotient of two such identities (say with $|T| = 1$ and $2$ respectively), we deduce that
  \[  \gamma(s, \sigma, \mathrm{spin}, \psi) = \gamma(s, \spin_*(\sigma), \mathrm{std}, \psi) \]
  as desired. This proves (i).

 \vskip 5pt
 
 For (ii), globalize $\chi$ to a quadratic Hecke character $\mathcal{\chi}$ of $\mathbb{A}_k^{\times}$. 
 By examining places outside $S$,  one checks readily that $\spin_*(\Sigma \otimes \mathcal{\chi})$ and $\spin_*(\Sigma) \otimes \mathcal{\chi}$ are nearly equivalent.
 By the strong multiplicity one theorem for $\GL_8$, it follows that $\spin_*(\sigma \otimes \chi) \cong  \spin_*(\sigma)\otimes \chi$, as desired.
\end{proof}

\vskip 10pt

 \section{\bf Results of Kret-Shin}  \label{S:KS}
 In this section, we briefly recall a special case of some recent results of Kret-Shin relevant for our applications. By combining the Spin lifting of the previous section and the results of Kret-Shin recalled here, we shall prove (vi) of the Main Theorem.
 \vskip 5pt

 \subsection{\bf Results of Kret-Shin}  \label{SS:KS}
 Assume that $\Sigma$ is a cuspidal automorphic representation of $\PGSp$ over a totally real number field $k$ such that
 \begin{itemize}
 \item $\Sigma_{\infty}$ is a sufficiently regular generic discrete series representation at each archimedean place $\infty$;
 \item $\Sigma_u$ is the Steinberg representation for some finite place $u$ of $k$.
\end{itemize}
Then Kret-Shin \cite[Thm. A]{KS} constructed a global Galois representation
\[ \begin{CD}
 \rho_{\Sigma}: {\rm Gal}(\overline{k}/k) @>>> \Spin_7(\overline{\Q}_l)   \end{CD} \]
satisfying:
\vskip 5pt

\begin{itemize}
\item for almost all places $v$, $\rho_{\Sigma}({\rm Frob}_v)_{ss}$ is conjugate to the Satake parameter $s_{\Sigma_v}$ in $\Spin_7(\overline{\Q}_l)$
\vskip 5pt

\item ${\std}\circ \rho_{\Sigma}$ is the Galois representation associated to the restriction ${\rm rest}(\Sigma)$ of $\Sigma$ to $\Sp$. 
\end{itemize}

\vskip 5pt

\subsection{\bf Spin lifting and Kret-Shin parameters}  \label{SS:Spin-KS}
We shall apply the results of Kret-Shin to the following problem. Suppose that $\sigma$ is an irreducible generic discrete series representation of $\PGSp$. Then
$\sigma$ determines a pair
\[  (\phi^{\flat}, \phi^{\#}) \in \Phi(\Sp) \times \Phi(\GL_8) \]
where  $\phi^{\flat}$ is the L-parameter of ${\rm rest}(\sigma)$  and $\phi^{\#}$ is the L-parameter of the Spin lifting  $\spin_*(\sigma)$   On the other hand, one has the natural map
\[  {\rm std}_* \times \spin_*: \Phi(\PGSp)  \longrightarrow    \Phi(\Sp)  \times \Phi(\GL_8) \]
and it is natural to ask if there exists $\phi \in \Phi(\PGSp)$  such that
\[  {\rm std}_*(\phi) = \phi^{\flat} \quad \text{and} \quad \spin_*(\phi) = \phi^{\#}. \]
Such a $\phi$ would be a natural candidate for an L-parameter of $\sigma$. 
The results of Kret-Shin allows us to show the existence of such a $\phi$. 
\vskip 5pt

\begin{prop} \label{P:KS}
Let $\sigma \in \Irr_{gen,ds}(\PGSp)$ be an irreducible generic discrete series representation. Then there exists $\phi \in \Phi_{ds}(\PGSp)$ such that
\vskip 5pt

\begin{itemize}
\item ${\rm std}_*(\phi)$ is the L-parameter of ${\rm rest}(\sigma)$ and
\item $\spin_*(\phi)$ is the L-parameter of $\spin_*(\sigma)$.
\end{itemize}
\end{prop}

  \vskip 5pt

 \begin{proof}  
Let $k$ be a totally real number field with a place $w$ such that $k_w \cong F$ and consider the group $\PGSp$ over $k$. 
By the globalization result in Proposition \ref{P:period_global_2} , we can find  a globally generic cuspidal  automorphic representation $\Sigma$ such that 
\vskip 5pt

\begin{itemize}
\item $\Sigma_w \cong \sigma$;
\item $\Sigma_u$ is the Steinberg representation for some finite place $u \ne w$
\item $\Sigma_{u'}$ is a generic supercuspidal representation at some other finite place $u'$ different from $w$ and $u$
\item  $\Sigma_{\infty}$ is a sufficiently regular generic discrete series representation for all archimedean places $\infty$. 
\end{itemize}
Such a $\Sigma$ is $L$-algebraic in the sense of \cite{BG}. 
\vskip 5pt

Now on one hand,  one may consider the global Spin lifting $\spin_*(\Sigma)$: this  is an isobaric automorphic representation of $\GL_8$ which is $L$-algebraic in the sense of \cite{BG}. On the other hand, 
by Kret-Shin \cite[Thm. A]{KS}, one can associate to $\Sigma$ a $\Spin_7(\overline{\Q}_l)$-valued  Galois representation $\rho_{\Sigma}$. Then the 8-dimensional representation $\spin \circ \rho_{\Sigma}$ is a Galois representation associated to $\spin_*(\Sigma)$ (by considering unramified places). 
 Let $\phi$ be the local L-parameter corresponding to the local Galois representation deduced from $\rho_{\Sigma}$ at the place $w$.  Then by local-global compatibility \cite{TY, Ca}, one sees readily that $\phi$ satisfies the requisite conditions of the proposition. 
 \end{proof}
 
 \vskip 5pt
 
 For a given $\sigma \in \Irr_{gen,ds}(\PGSp)$, we will call a $\phi$ which satisfies the conditions of Proposition \ref{P:KS} a Kret-Shin parameter of $\sigma$. Unfortunately, it is not true that $\sigma$ necessarily has a unique Kret-Shin parameter, because $\Spin_7(\C)$ is not acceptable \cite{CG1}. However, any two Kret-Shin parameter differs by twisting by quadratic characters. The following says that in some cases, the Kret-Shin parameter is unique.
 \vskip 5pt
 
 \begin{prop} \label{P:KS2}
 Assume that $\sigma \in \Irr_{gen, ds}(\PGSp)$ is the theta lift of a generic discrete series $\pi$ of $G_2$. Then the Kret-Shin parameter of $\sigma$ is unique (as an element of $\Phi(\PGSp)$) and is given by $\iota \circ \mathcal{L}(\pi)$.
 \end{prop}
 \vskip 5pt
 
 \begin{proof}
 Let $k$ be a number field as in the proof of Proposition \ref{P:KS} and consider the split group $G_2$ over $k$. By Proposition \ref{P:period_global_2} in Appendix A, we can find a globally generic cuspidal representation $\Pi$ such that $\Pi_w = \pi$ and $\Pi_u$ is the Steinberg representation for some finite place $u \ne w$. The global theta lift of $\Pi$ to $\PGSp$ is then a nonzero globally generic cuspidal representation $\Sigma$ of $\PGSp$. 
  As in the proof of Proposition \ref{P:KS}, by considering the global Spin lifting of $\Sigma$ to $\GL_8$ and looking at unramified places, we deduce by the strong multiplicity one theorem that  
 \[  \spin_*(\Sigma) = 1 \boxplus {\rm std}_*(\Sigma). \]
 This implies that the L-parameter of $\spin_*(\sigma)$ is $1 + \phi^{\flat}$.  On the other hand, it is also the case that
 \[   \spin \circ \iota \circ \mathcal{L}(\pi) =  1 + \phi^{\flat}. \]
Hence,  $\iota \circ \mathcal{L}(\pi)$ is a Kret-Shin parameter for $\sigma$. 
\vskip 5pt

Suppose $\phi$ is another Kret-Shin parameter for $\sigma$. Then $\phi$ is a quadratic twist of $\iota \circ\mathcal{L}(\pi)$ and $\spin \circ \phi = 1  \oplus \phi^{\flat}$. It follows by
 \cite[Lemma 4.6]{HKT} (which is a group theoretic result based on \cite{Gr95}) that  $\iota \circ \mathcal{L}(\pi)$ and $\phi$ are in fact conjugate in $\Spin_7(\C)$, i.e
\[   \phi = \iota \circ \mathcal{L}(\pi) \in \Phi(\PGSp).\]
This completes the proof of the proposition.
 \end{proof} 
 
 \vskip 5pt
 
 \begin{cor}  \label{C:six}
 One has the commutative diagram in (vi) of the Main Theorem.
 \end{cor}

 \section{\bf Surjectivity of $\mathcal{L}^{\diamondsuit}$}  \label{S:surj}
Using the results of the previous two sections,  we shall explain in this section how one can show the surjectivity of 
 \[  \mathcal{L}^{\diamondsuit}: \Irr_{ds}^{\diamondsuit}(G_2) \longrightarrow \Phi_{ds}^{\diamondsuit}(G_2). \]
 The main result of this section is:
 \vskip 5pt
  
 \begin{prop} \label{P:cuspidal_existence} 
 Let $\phi : WD_F \rightarrow G_2(\C)$ be a discrete series parameter belonging to $\Phi_{ds}^{\diamondsuit}(G_2)$.
 Then there exists a generic discrete series representation $\pi$ of $G_2$ such that $\mathcal{L}(\pi) = \phi$. 
 \end{prop} 
 \vskip 5pt
 
 \begin{proof} 
   If the L-parameter $\phi$ is nontrivial on the Deligne $\SL_2$ in $WD_F = W_F \times \SL_2(\C)$, then  the discussion in \cite[\S 3.5]{GS} gives us a candidate non-supercuspidal generic discrete series $\pi \in \Irr(G_2)$ with L-parameter $\phi$.  The results on theta correspondence in \cite{GS} allows us to determine $\theta(\pi) \in \Irr(\PGSp)$ explicitly, as a non-supercuspidal discrete series representation. Using the knowledge of the LLC for $\Sp$ (specifically the L-parameters of generic non-supercuspidal discrete series representations), one readily checks that $\mathcal{L}(\pi) = \phi$ as desired.
 
  \vskip 5pt
  
  Henceforth, we focus on the case when $\phi: W_F \longrightarrow G_2(\C)$ is trivial on the Deligne $\SL_2$.
Then the L-parameter $\iota \circ \phi$ is a supercuspidal  L-parameter of $\PGSp(F)$. Likewise,  $\phi^{\flat}:= {\rm std} \circ \iota \circ \phi $ is a supercuspidal L-parameter of $\Sp$ and its L-packet contains a unique generic supercuspidal  representation $\sigma^{\flat}$ with trivial central character.
Let $\sigma$ be any irreducible generic supercuspidal representation of $\PGSp$ such that the restriction to $\Sp$ contains $\sigma^{\flat}$.
The main issue is to show that we can pick $\sigma$ so that it  has nonzero theta lift to $G_2$. 
  
\vskip 10pt
Let $\phi_{\sigma}$ be a Kret-Shin parameter of $\sigma$ (as provided by Proposition \ref{P:KS}).  
Then one has
  \[ 
  {\rm std} \circ \phi_{\sigma} = {\rm std} \circ \iota \circ \phi, \]
     so that  there is a quadratic character $\chi$ such that $\iota\circ \phi  = \phi_{\sigma} \otimes \chi$.
     But by Proposition \ref{P:key}(ii), $\phi_{\sigma} \otimes \chi$ is a Kret-Shin parameter for $\sigma \otimes \chi$. Hence, 
      after replacing $\sigma$  by   $\sigma \otimes \chi$,  we may assume that $\iota \circ \phi$ is a Kret-Shin parameter for $\sigma$.
      In particular,  we have the following identities of local L-functions:
 \[   
 L(s, \sigma, \mathrm{spin}) = L(s, \spin_*(\sigma), \mathrm{std}) =
 L(s, \spin \circ \iota \circ\phi)  
 \] 
 with the first equality holding by  Proposition \ref{P:key}(i) and the second holding by Proposition \ref{P:KS} (and the property of the LLC for $\GL_8$).
 Since 
 \[  \spin \circ \iota \circ \varphi = 1 \oplus  {\rm std} \circ \iota \circ  \varphi  =1  \oplus \phi^{\flat}\]
 and thus  fixes a line in the 8-dimensional representation, we see that
 \[  L(0, \sigma, \mathrm{spin}) = L(0, \spin \circ \iota \circ\phi)  = \infty. \]
 Now it follows by \cite[Prop. 4.6 and Thm. 5.3]{SW} that $\sigma$ has nonzero theta lift $\pi$ on $G_2$, so that $\theta(\pi) = \sigma$ and
 $\mathcal{L}(\pi) = \phi$.    This completes the proof of the proposition.
   \end{proof}  
 
\vskip 10pt

\section{\bf Results of Bin Xu}  \label{S:Xu}
After the results of the previous sections, the remaining issue is the parametrization of the fibers of $\mathcal{L}^{\diamondsuit}$ over $\Phi_{ds}^{\diamondsuit}(G_2)$. 
For this, we shall appeal to recent results of Bin Xu \cite{Xu1,Xu2,Xu3} who   studied the problem of extending the LLC for the isometry groups ${\rm Sp}_{2n}$  to the corresponding similitude groups $\GSp_{2n}$, whose Langlands dual groups are $\GSpin_{2n+1}(\C)$. 
Let us briefly recall his results, specialized to the context of $n =3$.
\vskip 5pt

Take  an L-parameter $\phi^{\flat}$ for $\Sp$ with associated L-packet $\Pi_{\phi^{\flat}}$. 
Suppose that the representations in $\Pi_{\phi^{\flat}}$ has trivial central character (i.e. that $\phi^{\flat}$ can be lifted to $\Spin_7(\C)$).
We consider the finite set
\begin{equation} \label{E:flat}
 \tilde{\Pi}_{\phi^{\flat}}  =\{  \tilde{\pi} \in \Irr( \PGSp):   \,    {\rm rest}(\tilde{\pi}) \subset \Pi_{\phi^{\flat}} \}  \subset  \Irr (\PGSp). \end{equation} 
Observe that the group $\Hom(F^{\times}, \mu_2)$ of quadratic characters of $F^{\times}$ can be regarded as the group of quadratic characters of $\PGSp$ (by composition with the similitude map) and acts naturally on $\tilde{\Pi}_{\phi^{\flat}}$ by twisting. 

\vskip 5pt

 On the other hand,    consider the (finite)  set
 \begin{equation} \label{E:flat2} 
  \tilde{\Phi}_{\phi^{\flat}} := \{  \phi  \in \Phi(\PGSp) :    {\rm std}  \circ \phi = \phi^{\flat} \} \end{equation}
 of lifts of $\phi$  to $\Spin_7(\C)$.  The group $\Hom(W_F, \mu_2)$ of quadratic characters of $W_F$ acts transitively on $\tilde{\Phi}_{\phi^{\flat}}$ by twisting.  In particular, one has a canonical identification of component groups for any two elements of $\tilde{\Phi}_{\phi^{\flat}}$.
\vskip 5pt

 In \cite{Xu1, Xu2}, Xu has obtained a partition of $\tilde{\Pi}_{\phi^{\flat}}$ into a disjoint union of finite subsets satisfying the following list of properties, which is a compilation of 
 \cite[Prop. 6.27, Prop. 6.28 and  Thm. 6.30]{Xu1}, \cite[Prop. 4.4 and Thm. 4.6]{Xu2}) and \cite[Thm. 4.1]{Xu3}:
\vskip 5pt

\begin{itemize}
\item[(a)]    If $\tilde{\Pi}_{\phi^{\flat}}^X \subset \tilde{\Pi}_{\phi^{\flat}}$ denotes one such subset, then all others are of the form $\tilde{\Pi}_{\phi^{\flat}}^X \otimes \chi $ for a quadratic character $\chi$. We will call any of these subsets   a Xu's packet. 
\vskip 5pt

\item[(b)]   The natural restriction of representations of $\PGSp$ to $\Sp$ defines a bijection
\[  \tilde{\Pi}_{\phi^{\flat}}^X  \longrightarrow  \Pi_{\phi^{\flat}} / _{\PGSp}.  \] 
\vskip 5pt

\item[(c)]  There is a natural bijection
 \[  \tilde{\Pi}^X_{\phi^{\flat}} \longleftrightarrow \Irr (S_{\phi}/ Z(\Spin_7)) \]
  for any lift $\phi  \in \tilde{\Phi}_{\phi^{\flat}}$ with component group $S_{\phi}$.

\vskip 5pt

\item[(d)]  With respect to the parametrization above, the set $\tilde{\Pi}^X_{\phi,\nu}$ satisfies the stability properties and local character identities required by the theory of endoscopy. 

\vskip 5pt

\item[(e)] For any $\sigma \in \tilde{\Pi}^X_{\phi^{\flat}}$ and any $\phi \in \tilde{\Phi}_{\phi^{\flat}}$,  the stabilizer of $\sigma$  in  $\Hom(F^{\times}, \mu_2)$  is equal to the stabiliizer of  $\phi$ in $\Hom(W_F, \mu_2)$  (identified via local class field theory). 
 In particular,  
  \[  \# \{  \text{Xu's packets} \subset \tilde{\Pi}_{\phi^{\flat}}\}  =  \# \tilde{\Phi}_{\phi^{\flat}} \]
and the two sets above are (noncanonically)  isomorphic as homogeneous sets under $\Hom(F^{\times}, \mu_2) = \Hom(W_F, \mu_2)$.

\vskip 5pt

\item[(f)]  Globally, Xu used these local L-packets to describe the tempered part of the automorphic discrete spectrum of $\PGSp$ in the style of Arthur's conjecture, in terms of an Arthur multiplicity formula. 
\end{itemize}

 \vskip 5pt
  
 There is no doubt that Xu's packet $\tilde{\Pi}^X_{\phi^{\flat}}$ is an L-packet for $\PGSp$ and  its L-parameter should be an element of $\tilde{\Phi}_{\phi^{\flat}}$.
  Constructing an LLC for $\PGSp$ compatible with the restriction of representations to $\Sp$ amounts to  constructing  a bijection between these two sets.
  However, this problem was not resolved in \cite{Xu1, Xu2, Xu3}. We shall see in \S \ref{S:app} how this issue can be resolved if $\phi^{\flat}$ takes value in $G_2(\C)$. 
 \vskip 5pt

 \vskip 10pt

 \vskip 10pt
\section{\bf Fibers of $\mathcal{L}$}
Let $\phi \in \Phi_{ds}^{\diamondsuit}(G_2)$ be a discrete series parameter so that $\iota \circ \phi$ is a discrete series parameter of $\PGSp$. 
Then we have shown in Proposition \ref{P:cuspidal_existence} that $\mathcal{L}^{-1}(\phi)$ contains a generic discrete series representation and  is thus nonempty. In this section, we shall enumerate the fiber $\mathcal{L}^{-1}(\phi)$. 
\vskip 5pt

For ease of notation, let us set
\[  \phi^{\flat} = {\rm std} \circ \iota \circ \phi \quad \text{and} \quad \phi^{\#} = \spin \circ \iota \circ \phi. \]
 It is clear from the construction of $\mathcal{L}$ that 
 \[  \mathcal{L}^{-1}(\phi) \subset \tilde{\Pi}_{\phi^{\flat}}  \]
 and as we discussed in the previous section, the latter set is the disjoint union of Xu's packets. 
 \vskip 5pt
 
 \subsection{\bf Uniqueness of generic member}
 We first show:
 \vskip 5pt
 
 \begin{prop} \label{P:unique}
 For $\phi \in \Phi_{ds}^{\diamondsuit}(G_2)$, 
 there is a unique generic representation in $\mathcal{L}^{-1}(\phi)$.
  \end{prop}
 \vskip 5pt
 
 \begin{proof}
 Since the existence part has been shown in Proposition \ref{P:cuspidal_existence}, the main issue here is the uniqueness.
 Suppose, for the sake of contradiction,  that $\pi$ and $\pi'$ are two distinct generic representations of $G_2$ such that $\mathcal{L}(\pi) = \mathcal{L}(\pi') = \phi$. 
 Then $\pi$ and $\pi'$ belongs to $\Irr_{ds}^{\diamondsuit}(G_2)$ and we set
 \[  \sigma = \theta(\pi) \quad \text{and} \quad \sigma' = \theta(\pi'). \]
 Both $\sigma$ and $\sigma'$ are distinct generic discrete series representations of $\PGSp$ (by Howe duality) belonging to $\tilde{\Pi}_{\phi^{\flat}}$ (see (\ref{E:flat})).
 Hence, there is a quadratic character $\chi$ such that $\sigma' \cong \sigma \otimes \chi \ncong \sigma$. 
\vskip 5pt

Now by Proposition \ref{P:KS2}, $\sigma$   has (unique) Kret-Shin parameter $\iota \circ \phi$  which is an element of $\tilde{\Phi}_{\phi^{\flat}}$ (see (\ref{E:flat2}). 
On the other hand, 
\vskip 5pt

\begin{itemize}
\item by Proposition \ref{P:key}(ii), since $\sigma' \cong \sigma \otimes \chi$, $\sigma'$ has Kret-Shin parameter $(\iota \circ \phi) \otimes \chi$;
\item by Proposition \ref{P:KS2}, $\sigma' = \theta(\pi')$ has Kret-Shin parameter $\iota\circ \phi$.
\end{itemize}
By the uniqueness part of Proposition \ref{P:KS2}, we have
\[  \iota \circ \phi =   (\iota \circ \phi) \otimes \chi \in \tilde{\Phi}_{\phi^{\flat}}. \]
 Hence  the quadratic character $\chi$ stabilizes $\iota \circ \phi \in \tilde{\Phi}_{\phi^{\flat}}$ but does not stabilize $\sigma \in \tilde{\Pi}_{\phi^{\flat}}$. 
 This contradicts statement (e) of \S \ref{S:Xu}.
 \end{proof}
 \vskip 5pt
 
 \subsection{\bf A distinguished Xu's packet}   \label{SS:distin}
 In view of Proposition \ref{P:unique}, we can now pick out a distinguished Xu's packet contained in $\tilde{\Pi}_{\phi^{\flat}}$. Namely, we set
 \[  \tilde{\Pi}^{X\ast}_{\phi^{\flat}}  := \text{the Xu's packet containing the unique generic representation $\theta(\pi)$ with $\mathcal{L}(\pi) = \phi$}. \]
 Another way to say this is:
 \[  \tilde{\Pi}^{X\ast}_{\phi^{\flat}}  := \text{the Xu's packet containing the unique generic $\sigma \in \tilde{\Pi}_{\phi^{\flat}}$ with nonzero theta lift to $G_2$}. \]
 We shall show:

\vskip 5pt

\begin{prop} \label{P:fibers}
Let $\phi \in \Phi_{ds}^{\diamondsuit}(G_2)$.
The local theta correspondence defines a bijection 
\[  \mathcal{L}^{-1}(\phi) \longleftrightarrow  \tilde{\Pi}^{X\ast}_{\phi^{\flat}}. \]
\end{prop}
\vskip 5pt

\subsection{\bf Parametrization of fiber of $\mathcal{L}$}
Let us assume this proposition for a moment and explain how it leads to a proof of (vii) of the Main Theorem.
It was shown in \cite[Section 4, Prop. 1.10]{GrS2} that the inclusion $\iota$ gives rise to an isomorphism
\[  S_{\phi}  \cong S_{\iota\circ \phi}/ Z(\Spin_7), \]
and thus a bijection
\[  \Irr (S_{\iota\circ \phi}/ Z(\Spin_7)) \cong \Irr (S_{\phi}). \]
On the other hand, statement (c) of \S \ref{S:Xu} gives a bijection
\[  \tilde{\Pi}^{X\ast}_{\phi^{\flat}} \longleftrightarrow  \Irr (S_{\iota\circ \phi}/ Z(\Spin_7)). \]
Combining this with Proposition \ref{P:fibers}, we  obtain a bijection
\[   \mathcal{L}^{-1}(\phi) \longleftrightarrow   \tilde{\Pi}^{X\ast}_{\phi^{\flat}} \longleftrightarrow  \Irr (S_{\iota\circ \phi}/ Z(\Spin_7)) \longleftrightarrow \Irr (S_{\phi}). \]
This finishes the proof of  (vii) of the Main Theorem.
   
   \vskip  5pt 

\subsection{\bf One in, all in} 
It remains to prove Proposition \ref{P:fibers}. The following is the key lemma:

\vskip 5pt

\begin{lemma}  \label{L:allin}
Let $\tilde{\Pi}^X_{\phi^{\flat}} \subset \tilde{\Pi}_{\phi^{\flat}}$ be any Xu's packet. Then either all elements of $\tilde{\Pi}^X_{\phi^{\flat}}$ have nonzero theta lift to $G_2$ or  none of them has.
\end{lemma}
\vskip 5pt

\begin{proof}
It suffices to show that if $\sigma \in \tilde{\Pi}^X_{\phi^{\flat}}$ has nonzero theta lift to $G_2$, then so does any other $\sigma' \in \tilde{\Pi}^X_{\phi^{\flat}}$.
This will proceed by a global argument.

\vskip 5pt

Suppose that $\sigma = \theta(\pi)$ for $\pi \in \Irr^{\diamondsuit}_{ds}(G_2)$. Note that all non-supercuspidal representations in $\Irr^{\diamondsuit}_{ds}(G_2)$ are generic (see \cite[Thm. 15.3]{GS}). 
We shall globalize $\pi$ to a cuspidal representation $\Pi$ in one of the following two ways, depending on the type of $\pi$. 
\vskip 5pt
\begin{itemize}
\item[(i)] Suppose that $\pi$ is generic (this includes all non-supercuspidal $\pi$). Let $k$ be a totally real  number field with two places $w_1$ and $w_2$ such that $k_{w_1}\cong k_{w_2}\cong F$.  Let $\mathbb O$ be the split octonion algebra over $k$, and let $G=\Aut(\mathbb O)$. 
Fix an additional place $w$ and a 
generic cuspidal representation $\delta$ of $G(k_w)$ such that $\theta(\delta)$ is a generic cuspidal representation of $H(k_w)$.   
(Using \cite{G} there are such representations $\delta$ of depth 0.)  By Proposition  \ref{P:period_global_2} and Lemma \ref{L:isolation},  one can find a globally generic cuspidal automorphic representation $\Pi=\otimes_v \Pi_v$ of $G$ such that
\vskip 5pt

\begin{itemize} 
\item $\Pi_{\infty}$ is a discrete series representation, with the infinitesimal character sufficiently away from walls, for every archimedean place $\infty$. 
\item $\Pi_{w_1}\cong \Pi_{w_2}\cong \pi$     
\item $\Pi_w = \delta$;
\end{itemize} 
\vskip 5pt
Then the global theta lift  of $\Pi$ to $\PGSp$ is a globally generic cuspidal representation $\Sigma$  satisfying 
\vskip 5pt

\begin{itemize} 
\item $\Sigma_{\infty}$ is a discrete series representation for every archimedean place $\infty$. 
\item $\Sigma_{w_1}\cong \Sigma_{w_2}\cong \sigma = \theta(\pi)$. 
\item the restriction of $\Sigma$ to $\Sp$ is attached to  a  generic A-parameter. 
 \end{itemize} 
\vskip 5pt
Here, the first item  holds by the matching of infinitesimal characters \cite{HPS} and the fact that, at infinitesimal characters of discrete series sufficiently away from walls, only discrete series  representations are unitary. The third item  is assured by the first.

\vskip 5pt

\item[(ii)]  Suppose that $\pi$ is supercuspidal. Let $k$ be a totally real number field as in (i).  Let $\mathbb O$ be the totally definite  octonion algebra over $k$, and let $G=\Aut(\mathbb O)$.  By Proposition  \ref{P:period_global} and Corollary \ref{C:cuspidal}, one can find a cuspidal automorphic representation $\Pi$ such that:
\vskip 5pt
\begin{itemize}
\item $\Pi_{w_1} \cong \Pi_{w_2} \cong \pi$;
\item $\Pi_u$ is the Steinberg representation for some other finite place $u$;
\item $\Pi$ has nonzero cuspidal global theta lift $\Sigma$ on $\PGSp$.
\end{itemize}
 \vskip 5pt
 
  Thus, $\Sigma$ is a nonzero cuspidal automorphic representation of $\PGSp$ satisfying:
\vskip 5pt
\begin{itemize}
\item  $\Sigma_{w_1}\cong \Sigma_{w_2}\cong \sigma = \theta(\pi)$;
\item $\Sigma_u$ is the Steinberg representation;
\item the restriction of $\Sigma$ to $\Sp$  is attached to a generic A-parameter. 
\end{itemize}
 \end{itemize}
 Here the last item is a consequence of the second.
 \vskip 5pt 
Now let $\sigma'$ be a representation in $\tilde{\Pi}^X_{\phi^{\flat}}$. Then by the Arthur multiplicity formula for $\PGSp$ established by Xu, there exists a cuspidal
automorphic representation $\Sigma'$ satisfying:
\vskip 5pt

\begin{itemize}
\item   $\Sigma'_v \cong \Sigma_v$ for all $v \ne w_1$ or $w_2$;
\item $\Sigma'_{w_1} \cong \Sigma'_{w_2} \cong \sigma'$.
\end{itemize}
 In addition,  one has an equality of partial Spin L-functions:
 \[  L^S(s, \Sigma', \spin) = L^S(s, \Sigma, \spin) = \zeta^S(s) \cdot L^S(s, \Sigma, {\rm std}). 
 \]
 In particular, since $L^S(s, \Sigma, {\rm std})$ is nonzero at $s=1$ (as ${\rm rest}(\Sigma)$ is associated with a generic A-parameter), we see that $L^S(s, \Sigma', \spin)$ has a pole at $s=1$.   By \cite[Thm. 1.1]{GS_compositio},  $\Sigma'$ has nonzero global  theta lift 
to a form of $G_2$ over $k$. Specializing to the place $w_1$, we see that $\sigma'$ has nonzero local theta lift to the split $G_2$, as desired. This concludes the proof.
 \end{proof}
 \vskip 10pt
 
\subsection{\bf Proof of Proposition \ref{P:fibers} } 
Using Lemma \ref{L:allin}, we can now give the proof of Proposition \ref{P:fibers}. Lemma \ref{L:allin} implies that all members of the distinguished Xu's packet $\tilde{\Pi}^{X \ast}_{\phi^{\flat}}$ has nonzero theta lift to $G_2$ (since the generic member does). On the other hand, for any other Xu's packet contained in $\tilde{\Pi}_{\phi^{\flat}}$, the generic member does not participate in theta correspondence with $G_2$ by Proposition \ref{P:unique}. Hence, Lemma \ref{L:allin} implies that none of the members does. This proves Proposition \ref{P:fibers}.

\vskip 10pt

 \section{\bf Applications to Theta Correspondence}  \label{S:app}
 The starting point of this paper is the Howe duality and theta dichotomy result of \cite{GS} for the dual pairs
  
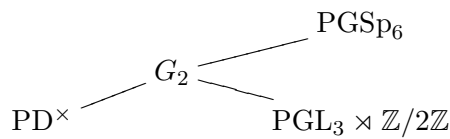
\[
\xymatrix@R=2pt{
&&{\rm PGSp}_6\\
&G_2 \ar@{-}[ru]\ar@{-}[ld]\ar@{-}[rd]&\\
{\rm PD^{\times}}&&{\rm PGL_3\rtimes \bbZ/2\bbZ}
}
\]
 The investigation of these dual pair correspondences were begun by Gross and the second author \cite{GrS1, GrS2}  in the 1990's, and precise conjectures  for the behaviour of these theta correspondences were formulated in \cite{GrS1, GrS2} in terms of the conjectural LLC for the groups involved. Our construction of the (refined) LLC for $G_2$ in this paper is to a large extent guided by these conjectures. With the results established in this paper, we can turn the table around and establish these conjectures from \cite{GrS1, GrS2}. 
 \vskip 5pt
 
 For this, we first need to establish a partial LLC for $\PGSp$. Namely, suppose that $\phi \in \Phi^{\diamondsuit}_{ds}(G_2)$  and we set 
 \[  \phi^{\flat} = {\rm std} \circ \iota \circ \phi \in \Phi_{ds}(\Sp). \]
 We have noted in \S \ref{S:Xu} that the two sets
   \[  \{  \text{Xu's packets} \subset \tilde{\Pi}_{\phi^{\flat}}\}   \quad \text{and} \quad \tilde{\Phi}_{\phi^{\flat}} \]
 are noncanonically isomorphic as $\Hom(F^{\times}, \mu_2)$-sets.  However, we have seen in \S \ref{SS:distin} that one has a distinguished Xu's packet $\tilde{\Pi}^{X\ast}_{\phi^{\flat}}$ in the first set. Moreover, the second set contains the distinguished L-parameter $\iota \circ \phi$.  The existence of these distinguished base points gives rise to a canonical isomorphism of $\Hom(F^{\times} ,\mu_2)$-sets.
 \vskip 5pt
 
   In other words, we define the L-parameter of the distinguished Xu's packet $\tilde{\Pi}^{X\ast}_{\phi^{\flat}}$ to be $\iota \circ \phi$.
  This definition is not really ad hoc. Indeed, the distinguished Xu's packet $\tilde{\Pi}^{X\ast}_{\phi^{\flat}}$ is the unique one in  $ \tilde{\Phi}_{\phi^{\flat}}$ for which the Langlands-Shahidi Spin L-function of  its unique generic member is equal to the local L-factor of $\spin \circ \iota \circ \phi$ (and hence has a pole at $s = 0$).  
 Further,   by extending the spin lifting constructed in \S \ref{S:triality} and  using the results of Kret-Shin recalled in \S \ref{SS:KS}, we construct in Appendix C below  a weak LLC for $\PGSp$, of which this particular L-packet is an instance. 
  \vskip 5pt

  In any case,  with this definition of the L-packet of  $\PGSp$ associated to $\iota \circ \phi$, and the results of \cite{GS} as well as the LLC of $G_2$ established in this paper, we have:
  \vskip 5pt
  
  \begin{thm}
  (i) Conjecture 3.1 in \cite{GrS1} holds when $p \ne 3$.
  \vskip 5pt
  
  (ii) Conjecture 1.13 (as well as Conjectures 2.3 and 2.5) in \cite[Section 4]{GrS2} hold. 
  \end{thm} 
 Thus, in a sense, we have come full circle. 
 \vskip 10pt
 
 \section{\bf Appendix A: Some Globalization Results}  \label{S:globalization}
In this appendix, we record some globalization results that were used in the main body of the paper.
\vskip 10pt

\subsection{A local non-vanishing result} \label{SS:local}
Let $F$ be a $p$-adic field and $\bbO$ be an octonion algebra over $F$, and let 
 $G=\Aut(\bbO)$. Let $D$ be a quaternion algebra over $F$ and fix an embedding $D \hookrightarrow \mathbb{O}$, unique up to conjugation by $G$. 
 Let $G_D\subset G$ be the pointwise stabilizer of $D\subset \bbO$. 
   Then $G_D$ is isomorphic to the group of norm one elements in $D$.  
\vskip 5pt

Let  $H= \PGSp(F)$, and $P=MN$ the Siegel maximal parabolic subgroup. Let $\hat N$ be the unitary dual of the unipotent radical $N$.  
The group $M$ acts on $\hat N$ with open orbits parameterized by quaternion algebras $D$: the stabilizer in $M$ of an element in the orbit parameterized by
is isomorphic to $\Aut(D)$. 
\vskip 5pt

Recall that we have the dual pair $G \times H \subset E_7$ and
the minimal representation $\Pi$ of $E_7$ induces a local theta correspondence $\theta: \Irr^{\heartsuit}(G) \longrightarrow \Irr(H)$.
Now we note:
\vskip 5pt

\begin{lemma} \label{L:su_period} 
Let $\pi$ be an irreducible representation of $G$ such that $\theta(\pi)$ is a square integrable representation of $H$. Then there exists 
a quaternion algebra $D \subset \bbO$ such that  $\pi$ is a quotient of $\mathrm{ind}^G_{G_D}(1)$. 

\end{lemma} 
\begin{proof} Let $c_{v,u}(g) =\langle\theta(\pi)(g) v, u\rangle  $ be a matrix coefficient of $\theta(\pi)$.  For some choices of $u$ and $v$ in $\theta(\pi)$, the restriction of $c_{v,u}$ to $N$ is non-zero. 
Since $\theta(\pi)$ is square integrable, the restriction of $c_{v,u}$ to $N$ is integrable. Thus, for every 
$\psi_N\in \hat N$, the following integral is well defined
\[ 
\hat c_{u,v}(\psi_N) = \int_N  c_{v,u}(n)\cdot \overline{\psi_N(n)} \, dn. 
\] 
Furthermore,  $\hat c_{u,v}$ is a continuous function on $\hat N$ by standard Fourier analysis arguments. It follows that there exists $\psi_N$, in one of the open orbits, such that 
$\hat c_{u,v}(\psi_N) \neq 0$.  As we fix $u$ and vary $v$, $v\mapsto \hat c_{u,v}(\psi_N) $ is a linear functional on  $\theta(\pi)_{N, \psi_N}$. Thus  $\theta(\pi)_{N, \psi_N} \neq 0$.  
\vskip 5pt

Let $\Pi$ be the minimal representation of the group $E_7$, so that 
$\pi\otimes \theta(\pi)$ is its quotient. From \cite[Section 5, Lemma 3.4]{GrS2} (case (4) in the proof) we have 
\[ 
\Pi_{N, \psi_N}= \mathrm{ind}^G_{G_D}(1) 
\] 
where $D$ is such that $\psi_N$ belongs to the open orbit parameterized by $D$. The lemma follows. 
\end{proof} 

\begin{cor} \label{C:su_period} 
 Let $\pi$ be a supercuspidal representation  of $G$ such that $\theta(\pi)$ is a square integrable representation of $H$. Then there exists a quaternion algebra $D$ such that 
$\pi$ is a submodule of $\mathrm{ind}^G_{G_D}(1)$.  
\end{cor} 
\begin{proof} It follows from projectivity of $\pi$.  
\end{proof} 

\vskip 5pt

\begin{lemma} \label{L:su_period_St}
Let $\mathrm{St}$ be the Steinberg  representation of $G$. Let $D$ be a division quaternion algebra. Then 
$\mathrm{St}$  is a direct summand of $L^2(G_D \backslash G)$. 
\end{lemma}
\begin{proof} 
Recall that $\theta(\mathrm{St})$ is the Steinberg representation of $H$. By \cite[Prop. 5]{CS}, the restriction of the Steinberg representation of $H$ to any maximal unipotent subgroup of $H$  
 is the regular representation. It follows at once that $\theta(\mathrm{St})_{N,\psi_N} \neq 0$ for any character $\psi_N$ of $N$. Hence, arguing as in Lemma \ref{L:su_period}, $\mathrm{St}$ is a 
 quotient of $\mathrm{ind}^G_{G_D}(1)$ with $D$ non-split.  Since 
\[ 
0\neq  \Hom_G( \mathrm{ind}^G_{G_D}(1), \mathrm{St}) \cong  \Hom_G(\mathrm{St},  \mathrm{Ind}^G_{G_D}(1)) \cong  \Hom_{G_D}( \mathrm{St}, \mathbb C). 
\] 
and $G_D$ is anisotropic, there exists a non-zero vector in $\mathrm{St}$ fixed by $G_D$. The corresponding matrix coefficient gives us 
\[ \mathrm{St} \subset L^2(G_D \backslash G)\subset L^2(G).\]

\end{proof}

\vskip 5pt

We have a complementary result for the compact Lie group $G^c_2(\mathbb{R})$ which is the automorphism group of the octonion division $\mathbb{R}$-algebra.
\vskip 5pt

\begin{lemma} \label{L:real_period}
Let $\mathbb H$ be the quaternion division $\mathbb{R}$-algebra. For any irreducible (finite-dimensional) representation $\pi$ of $G^c_2(\mathbb{R})$, one has 
$\pi^{G_{\mathbb H}(\mathbb{R})} \ne 0$, so that 
\[  \pi \subset L^2(G_{\mathbb H}(\mathbb{R}) \backslash G^c_2(\mathbb{R})). \]
\end{lemma}
\vskip 5pt
\begin{proof}
One has $G_{\mathbb H}(\mathbb{R}) \subset {\rm SU}_3 \subset G^c_2(\mathbb{R})$.  By the Gelfand-Zetlin branching rule, any irreducible representation of $\SU_3$ has nonzero vectors fixed by $G_{\mathbb H}(\mathbb{R}) \cong {\rm SU}_2$.  The proposition follows.  
\end{proof}

\subsection{First globalization result}  \label{SS:global}
Let $k$ be a totally real number field and $\mathbb O$ be a totally definite octonion algebra over $k$, so that 
$\mathbb O_{\infty}$ is non-split for all archimedean places $\infty$ of $k$, and set $G=\Aut(\bbO)$. For any place $v$ of $k$, write  $G_v=G(k_v)$ for simplicity.  Then $G_v$ is 
anisotropic for archimedean places $v$ and split otherwise.  


\vskip 5pt 
Let $D\subseteq \mathbb O$ be a quaternion subalgebra. Define 
\[ 
L^2_{D}(G(k)\backslash G(\mathbb A)) \subset L^2(G(k)\backslash G(\mathbb A))
\] 
be the Hilbert subspace which is orthogonal to the span of all automorphic representations $\Pi=\otimes_v \Pi_v$ such that  
\[ 
\int_{G_D(k)\backslash G_D(\mathbb A)} f(g) ~dg = 0 
\] 
for all $f\in \Pi$. Hence, for any $\Pi \subset L^2_{D}(G(k)\backslash G(\mathbb A))$, the global period integral over $G_D$ is nonzero on $\Pi$.
\vskip 5pt

\begin{prop} \label{P:period_global}  
  Let $v_1, \ldots , v_n$ be a (nonempty) set of places of $k$. 
 For nonarchimedean $v_i$, let $\pi_i$ be either the Steinberg representation ${\rm St}$  or a supercuspidal representation of $G_{v_i}$ such that $\theta(\pi_i) \in \Irr(H_{v_i})$ is square integrable. For a real $v_i$, let $\pi_i$ be any irreducible (finite dimensional) representation. 
 Then there exists  a quaternion algebra $D \subseteq \mathbb{O}$ and 
an automorphic representation $\Pi=\otimes_v \Pi_v$ in $L^2_{D}(G(k)\backslash G(\mathbb A))$ 
such that  $\Pi_{v_i}\cong \pi_i$ for $i=1, \ldots, n$.   
\end{prop} 
\vskip 5pt

\begin{proof}  Pick $D$ such that for all $i=1, \ldots, n$,  
 \[  \pi_i \subset L^2(G_{D,v_i} \backslash G_{v_i} ). \] 
 This is possible  by Corollary \ref{C:su_period},  Lemma \ref{L:su_period_St} and Lemma \ref{L:real_period}.
 We can now proceed  following  \cite[Theorem 16.3.2  and Remark 16.4.1]{Sak_Ven}. 
 Using Poincar\'e series arising from compactly supported functions on $G_D\backslash G$, 
 one shows that 
\[ 
\bigotimes_{i=1}^n L^2(G_{D, v_i} \backslash G_{v_i})  \quad \text{is weakly contained in} \quad L^2_{D} (G(k)\backslash G(\mathbb A)) \]
in the sense of  Definition \ref{D:weak}.   
Hence,  $\otimes_i \pi_i$ is  weakly contained there as well. 
 The proposition now follows from 
Corollary \ref{C:weak} , since supercuspidal representations and $\mathrm{St}$   are isolated  in the unitary dual of $G_{v_i}$ by Proposition \ref{P:isolated} 
(for the Steinberg representation).  
\end{proof}
\vskip 5pt  

We consider the global theta correspondence for the dual pair $G \times H$ 
in the adjoint group of absolute type $E_7$ (and $k$-rank $3$), that corresponds to the Albert algebra $J_3(\mathbb O)$ via the Koecher-Tits construction,  with respect to the minimal representation 
constructed in \cite[Theorem 6.4]{HS}.   

\begin{cor}  \label{C:cuspidal}
 In the context of Proposition \ref{P:period_global}, the representation $\Pi$ has nonzero global theta lift to $H = \PGSp$. 
 Moreover, this global theta lift is cuspidal if one of the following conditions hold:
 \vskip 5pt
 
 \begin{itemize}
 \item at some nonarchimedean $v_i$, $\pi_i$ is the Steinberg representation;
 \item at a real $v_i$, $\pi_i$ has regular highest weight. 
 \end{itemize}
 \end{cor}
 \vskip 5pt
 
 \begin{proof} 
The nonvanishing of the global theta lift follows by  \cite[Chapter 5, Prop. 4.5]{GrS2}: it is for this global nonvanishing result that we insist on globalizing with a nonzero $G_D$-period in Proposition \ref{P:period_global}. The cuspidality of the global theta lift follows by \cite[Chapter 5, Cor. 4.9]{GrS2}. We note that though \cite{GrS2} works over the base field $\Q$, the results from \cite{GrS2} that we use in this proof hold over a general number field $k$ with the same proofs given there.  
\end{proof}

\vskip 10pt
\subsection{Second globalization result}  \label{SS:global_2}

Assume now that $k$ is an arbitrary number field and $G$ is a simple split group over $k$. Let $U$ be the unipotent radical of a Borel subgroup of $G$, and fix a Whittaker character $\psi = \prod_v \psi_v:  U(\mathbb A)/U(k) \rightarrow \mathbb C^{\times}$. 
In particular, we have a notion of  local and  automorphic Whittaker functional (relative to $\psi$) and the notion of (Whittaker) generic representations, both locally and globally. 
\vskip 5pt

Fix a place $w$ of $k$ and a $\psi_w$-generic supercuspidal representation $\sigma$ of $G(k_w)$.   Let 
\[ 
L^2_{\psi-\mathrm{gen}, \sigma} (G(k)\backslash G(\mathbb A)) \subset L_{cusp}^2(G(k)\backslash G(\mathbb A))
\] 
be the Hilbert subspace of the automorphic discrete spectrum of $G$ which is orthogonal to the span of all square-integrable  automorphic representations $\Pi=\otimes_v \Pi_v$ 
such that either the automorphic $\psi$-Whittaker functional vanishes on $\Pi$ or    $\Pi_w\ncong \sigma$. 
Then any $\Pi \subset L^2_{\psi-\mathrm{gen}, \sigma} (G(k)\backslash G(\mathbb A))$ is globally (Whittaker) generic and $\Pi_w \cong \sigma$.

\vskip 5pt 
For any place $v$, let $L^2_{\psi_v}(N(k_v)\backslash G(k_v))$ be the space of square integrable Whittaker functions. Recall that all generic square integrable representations of $G(k_v)$ are 
direct summands of $L^2_{\psi_v}(N(k_v)\backslash G(k_v))$ \cite{Sak_Ven}.

\begin{prop}  \label{P:period_global_2} 
Let $v_1, \ldots , v_n$ be a set of  (possibly archimedean) places of $k$, different from $w$. For $i=1, \ldots, n$, let $\delta_i$ be a $\psi_{v_i}$-generic square integrable representation of $G(k_{v_i})$. Assume that  these representations are isolated in $ L^2_{\psi-\mathrm{gen}, \sigma} (G(k)\backslash G(\mathbb A))$. Then there exists an irreducible cuspidal automorphic representation $\Pi$ in 
$ L^2_{\psi-\mathrm{gen}, \sigma} (G(k)\backslash G(\mathbb A))$  such that $\Pi_{v_i} \cong \delta_i$ for $i=1, \ldots, n$.  Moreover, we may assume that for all  finite places $v$ different from $w$ or $v_i$, $\Pi_v$ is a principal series representation induced from a Borel subgroup. 
\end{prop} 
\vskip 5pt

\begin{proof} Again, using Poincar\'e series \cite[Theorem 16.3.2  and Remark 16.4.1]{Sak_Ven}, one proves that 
\[ 
L^2_{\psi_{v_1}}(N(k_{v_1})\backslash G(k_{v_1}))\otimes \cdots \otimes L^2_{\psi_{v_n}}(N(k_{v_n})\backslash G(k_{v_n}))
\] 
is weakly contained in $L^2_{\psi-\mathrm{gen}, \sigma} (G(k)\backslash G(\mathbb A))$. In particular, $\delta_1\otimes \cdots \otimes \delta_n$ is weakly contained there as well.  
The existence of the globalization $\Pi$  now follows from Proposition \ref{P:weak} and the hypothesis of isolation of the $\delta_i$'s.   The extra control at finite places different from $w$ and $v_i$ is achieved by using appropriate test functions  at these other places  as the input for the Poincar\'e series; see  \cite[proof of Prop. A.2]{Gan_Ich}.

\end{proof}

\vskip 5pt

In applications of Proposition \ref{P:period_global_2}, one would need to verify the  hypothesis of isolation  in the proposition. For classical groups and their associated similitude
 groups, this isolation is a consequence of estimates towards the Ramanujan conjecture for globally generic cuspidal representations; see \cite[Prop. A.5 and Lemma A.2]{ILM}. The following verifies it for the group $G_2$.
 \vskip 5pt
 
 \begin{lemma} \label{L:isolation}  
In the context of Proposition \ref{P:period_global_2}, assume  that $\sigma$ is a generic cuspidal representation of $G_2(k_w)$ such that its local theta lift to $\PGSp(k_w)$ is a generic supercuspidal representation (using \cite{G}, there are such representations $\delta$ of depth 0).
 Then $\otimes_i \delta_i$ is isolated in $ L^2_{\psi-\mathrm{gen}, \sigma} (G(k)\backslash G(\mathbb A))$ with respect to the Fell topology. 
\end{lemma} 
\vskip 5pt

\begin{proof} 
 Any irreducible $\Pi \subset L^2_{\mathrm{gen}, \sigma} (G_2(k)\backslash G_2(\mathbb A))$ has a nonzero globally generic  global theta lift 
$\Sigma$ (see  \cite[appendix]{HKT}. The hypothesis on $\sigma$ implies that $\Sigma$ is cuspidal. By the approximation to the Ramanujan conjecture for $\PGSp$, and our explicit results on local theta correspondence (including the correspondence of infinitesimal character at archimedean places), we conclude as in \cite[proof of Prop. A.2]{Gan_Ich} that for each $1\leq i \leq n$. 
\[  \Pi_{v_i} \in \Irr_{unit,\mathrm{gen}, \leq c}  (G_2), \]
where the latter is the set of generic irreducible unitary representations whose exponents are bounded by $c$. Hence,
\[  L^2_{\mathrm{gen}, \sigma} (G_2(k)\backslash G_2(\mathbb A)) \quad \text{is weakly contained in  $\Irr_{unit, \mathrm{gen}, \leq c}  (G_2)$} \]
as a representation of $G_2(k_{v_i})$.
 It follows by \cite[Lemma A.2]{ILM} that $\delta_i$ is isolated in $\Irr_{unit, \mathrm{gen}, \leq c}  (G_2)$ and hence in  $L^2_{\mathrm{gen}, \sigma} (G_2(k)\backslash G_2(\mathbb A))$ as well.
  \end{proof}

\vskip 10pt

 \section{\bf Appendix B: Topology of Unitary dual}  \label{S:appB}

Let $G$ be a (simple) Chevalley group over a $p$-adic field $F$.  Recall that a unitary representation $\pi$ of $G$ on a Hilbert space $H_{\pi}$ is continuous if, for every $v\in H_{\pi}$, 
the map $g\mapsto \pi(g) v$, where $g\in G$, is continuous.  In that case the subspace $V_{\pi} \subseteq H_{\pi}$ of smooth vectors is dense.  All unitary representations are assumed 
continous. 
We say that a smooth irreducible representation of $G$ is 
unitarizable, if it has a non-trivial $G$-invariant hermitian product or, equivalently, it is the space of smooth vectors in an irreducible unitary representation.

\subsection{Steinberg is isolated} Using an argument due to Mili\v ci\'c,  we prove that the Steinberg representation of $G$ 
is isolated, with respect of the Fell topology,  in the unitary dual of $G$. 
 (We refer the reader to \cite[Appendix F]{BHV}, for a nice summary of Fell topology, definitions and basic properties).  Here we work with smooth representations. 
 
 \vskip 5pt 
 An irreducible unitarizable representation $\pi$ of $G$ is a limit of a sequence of irreducible  unitarizable representations $\pi_n$ of $G$ if for one or, equivalently, 
any matrix coefficient $\langle \pi(g) v, v\rangle$ of $\pi$ there is a sequence of matrix coefficients $\langle \pi_n(g) v_n, v_n\rangle$ of $\pi_n$ converging uniformly on all compact sets in $G$ to 
$\langle \pi(g) v, v\rangle$.  Fix a special maximal compact subgroup $K$ in $G$. Then $K$ has a quotient $G(\mathbb F_q)$ where $\mathbb F_q$ is the residual field of $F$. 

\begin{prop}  \label{P:isolated} 
Let $\sigma$ be the Steinberg representation of $G$. If the rank of $G$ is at least 2, then $\sigma$ is isolated in 
the unitary dual of $G$. 
\end{prop} 
\begin{proof}  
Let $\tau$ be any smooth irreducible representation of $K$. Let $f_{\tau}$ be a smooth function on $G$, supported on $K$, and equal to the complex conjugate 
of the character of $\tau$, divided by the volume of $K$, on $K$.  Thus, for any unitary representation $(\pi,V)$ of $G$, 
the operator $\pi(f_{\tau})$ is the projector on the $\tau$-isotypic subspace of $V$.   

\begin{lemma} \label{L:types} 
Let $\pi$ be an irreducible unitarizable representation such that the $K$-type $\tau$ occurs in $\pi$. Assume that $\pi$ is a limit of a sequence of 
irreducible representations $\pi_n$. Then, for almost all $n$, the $K$-type $\tau$ occurs in $\pi_n$.  
\end{lemma} 
\begin{proof} By the assumption there exists a non-zero $v\in V$, such that  $\langle \pi(f_{\tau}) v, v\rangle = \langle v, v\rangle \neq 0$. Let 
$\langle \pi_n(g) v_n, v_n\rangle$ be a sequence of matrix coefficients of $\pi_n$ converging uniformly on all compact sets in $G$, in particular on $K$, to 
$\langle \pi(g) v, v\rangle$. Thus the limit of $\langle \pi_n(f_{\tau}) v_n, v_n\rangle$ is $\langle \pi(f_{\tau}) v, v\rangle$, and the lemma follows.  
\end{proof} 

Next, we need the following lemma which, for real groups, was established by Vogan \cite{Vo} in a full generality.  

\begin{lemma} \label{L:vogan} 
Let $\sigma$ be the Steinberg representation of $G$. Let $\pi_n$ be a sequence of irreducible unitary representations $\pi_n$ of $G$ converging to $\sigma$. Then, after passing to a 
subsequence of $\pi_n$, if necessary, there exists a parabolic subgroup $P=MN$ of $G$, and a sequence of characters $\mu_n$ of $M$, converging to a character $\mu_0$, such that 
\begin{itemize} 
\item $\pi_n\cong \Ind_P^G(\sigma_M\otimes \mu_n)$, where $\sigma_M$ is the Steinberg representation of $M$. 
\item $\sigma$ is a subquotient of $\Ind_P^G(\sigma_M\otimes \mu_0)$. 
\end{itemize} 
\end{lemma} 
\begin{proof}  Fix a Borel subgroup $B=TU$ in $G$, where $U$ is the unipotent radical of $B$, and $T$ is a maximal split torus of $G$. 
 Let $\Delta$ be the corresponding set of simple roots. By Lemma \ref{L:types} we can assume, with out loss of generality, that all $\pi_n$ have 
Iwahori-fixed vectors.  It follows that all $\pi_n$ are subquotients of $\Ind_B^G(\chi_n)$ for some unramified characters $\chi_n$ of $T$. Here, as usually, the induction is normalized. 
Let  $\mathrm{st}$  be the inflation to $K$ of the Steinberg representation of the finite group $G(\mathbb F_q)$. 
The restriction of $\sigma$ to $K$ contains  $\mathrm{st}$. It follows that (almost all) all $\pi_n$ contain that $K$-type. 
\vskip 5pt 

We need to recall some well known facts on induced representations $\Ind_B^G(\chi)$  where $\chi$ is an unramified character of $T$. The first is that the 
type $\mathrm{st}$ appears with multiplicity one, a simple check using the Cartan decomposition $G=BK$.  The second concerns decomposition of 
$\Ind_B^G(\chi)$:  if the character $\chi$ is regular, then irreducible suquotients of $\Ind_B^G(\chi)$ are understood by a result of Rodier. We shall use these two to
pin-point the unique irreducible subquotient of $\Ind_B^G(\chi)$ containing the type $\mathrm{st}$.  
 More precisely,  for any $\alpha\in \Delta$, let $\alpha^{\vee}: F^{\times} \rightarrow T$ be the corresponding co-root. 
We can write $\chi\circ \alpha^{\vee} = |\cdot |^{m_{\alpha}}$ 
for some complex numbers $m_{\alpha}$, where the real part $\Re (m_{\alpha})$ is well defined.  
Assume that $\Re (m_{\alpha})>  0$ for all $\alpha\in \Delta$, in particular, the character $\chi$ is regular. Let $S\subseteq \Delta$ be the set of $\alpha$ such that $m_{\alpha}=1$.  
Let $P=MN\supseteq B$ be the unique parabolic subgroup of $G$ corresponding to $S$. Then $\Ind_B^G(\chi)$ contains $\Ind_P^G(\sigma_M\otimes \mu)$ for a character $\mu$ of $M$ that is easy to 
determine from $\chi$.  For example, if $S=\Delta$ then $P=G$ and this statement is the well known fact that a character twist of the Steinberg representation is contained in $\Ind_B^G(\chi)$. 
 The general case follows by induction in 
stages. The representation $\Ind_P^G(\sigma_M \otimes \mu)$ contains the type $\mathrm{st}$, again an easy argument using induction in stages, and it is irreducible by 
the result of Rodier \cite{Ro}.  

\vskip 5pt 

Since, for any $w$ in the Weyl group, 
$\Ind_B^G(\chi)$ and $\Ind_B^G(\chi^w)$  have the same subquotients, we can assume that for all $\chi_n$ we have $\Re (m_{\alpha})\geq 0$, for $m_{\alpha}$ corresponding to $\chi_n$.  
Moreover, since $\sigma$ is the limit of $\pi_n$ and  $\sigma$ is contained in $\Ind_B^G(\chi_0)$, where where $\chi_0$ satisfies $\chi_0\circ \alpha^{\vee} = |\cdot|$ for all $\alpha \in \Delta$, 
we can assume that $\Re (m_{\alpha})>  0$ for all $\chi_n$.  
Now, as we just discussed, each $\pi_n$ is isomorphic to $\Ind_P^G(\sigma_M, \otimes \mu)$  for a parabolic $P$ containing $B$. Since there are finitely many parabolic 
subgroups containing $B$, the lemma follows.

\end{proof} 

Now we can easily finish the proof of proposition.  Let $r$ be the rank of $G$. Then $\Ind_B^G(\chi_0)$ has $2^r$ irreducible suqbquotients. Their exponents are known, and thus 
the asymptotics of their matrix coeffcients. It follows that the matrix coefficients of subquotients, different from the Steinberg representation, are not decaying at infinity. Thus, by Howe-Moore 
the only unitarizable subquotients of $\Ind_B^G(\chi_0)$ are the trival and the Steinberg representations. Thus, if $r> 1$  then 
$\Ind_P^G(\sigma_M\otimes \mu_0)$ has a non-unitary subquotient. On the other hand, since $\Ind_P^G(\sigma_M\otimes \mu_0)$ is the limit of the unitary sequence $\pi_n$, all its subquotients 
are unitary, by a result of Mili\v ci\' c \cite[Thm. 6]{Mi}; see also \cite{Ta1}. This is a contradiction. 
\end{proof} 

\vskip 5pt 

\subsection{Weak containment} 


\vskip 5pt 
\begin{definition} \label{D:weak} 
Let $\pi$ and $\sigma$ be two unitary representations of $G$. We say that $\sigma$ is weakly contained in $\pi$ if every normalized 
matrix coefficient $\langle\sigma(g) v,v\rangle$ of $\sigma$, that is $||v||=1$,  can be approximated uniformly, on compact subsets of $G$, by convex combinations of normalized 
matrix coefficients of $\pi$.   
\end{definition}

\begin{prop} \label{P:weak} 
 Let $\pi_1, \pi_2, \ldots$ be a sequence of irreducible unitary representations of $G$. Let $\pi = \hat\oplus_{i=1}^{\infty} \pi_i$. Let 
$\sigma$ be an irreducible unitary representation of $G$. If $\sigma$ is weakly contained in $\pi$ then 
$\sigma$ is a limit of a subsequence of $\pi_n$.  
\end{prop} 
\begin{proof}
Recall that 
$L^{\infty}(G)$ is the dual space of $L^1(G)$, and the unit ball in $L^{\infty}(G)$ is compact in the weak-$\ast$ topology. 
Let $\mathcal C$ be the weak-$\ast$ closure, in $L^{\infty}(G)$, of the convex hull of 
\[ 
\mathcal F =\{ \langle \pi_n(g) v_n, v_n \rangle ~| ~ v_n \in \pi_n, ||v_n||=1, n=1,2 \ldots \}.  
\] 
Then $\mathcal C$ is closed and convex subset of the unit ball, in particular, it is  compact by Alaoglu's Theorem \cite[page 130]{Co}. 
Recall that an element of a convex set is extremal if it is not a proper convex combination of two other 
elements of the set. By   \cite[Thm. C.5.2]{BHV}, 
elements of $\mathcal F$ are extremal in $\mathcal C$.  (That theorem states that normalized matrix coefficients associated to an irreducible representation are 
extremal in a larger convex set.)  By our assumption, $f(g)= \langle \sigma(g) v, v \rangle$, for $v\in \sigma$ of norm 1, is contained in $\mathcal C$. Thus it is also an extremal point of $\mathcal C$ 
since $\sigma$ is irreducible. 
By the Krein-Milman Theorem \cite[page 141]{Co}, any extremal point in $\mathcal C$ is in the 
closure of $\mathcal F$.  In particular, $f$ is 
a limit of a sequence $f_n(g)=\langle \pi_n(g) v_n, v_n \rangle$ in $\mathcal F$.  Of course, the limit is in the  weak-$\ast$ topology, however, by a theorem of Raikov  
\cite[Thm. C.5.6]{BHV}, 
the sequence $f_n$ converges uniformly on compact sets to $f$.  This is banal for $p$-adic groups. Indeed, we can assume that $v$ and all $v_n$ are $K$-fixed for a small 
open compact subgroup of  $G$.  Thus these functions are constant on $K$-double cosets, and uniform convergence on compact sets is equivalent to point-wise convergence. Thus, if $g\in G$ 
and $\chi_g$ is the characteristic function of $KgK$ divided by the volume of $KgK$ volume,  then we have 
\[ 
f_n(g) = \int_G f_n \cdot  \chi_{g}  \rightarrow \int_G f \cdot  \chi_{g} = f(g) 
\] 
since $f_n \rightarrow f$ in the weak-$\ast$ topology.  

\end{proof}

\begin{cor} \label{C:weak} 
 Let $\pi_1, \pi_2, \ldots$ be a sequence of irreducible unitary representations of $G$. Let $\pi = \hat\oplus_{i=1}^{\infty} \pi_i$. Let 
$\sigma$ be an irreducible unitary representation of $G$ isolated in the unitary dual. If $\sigma$ is weakly contained in $\pi$ then 
$\sigma\cong \pi_i$ for some $i$.  
\end{cor}

\vskip 10pt

\section{\bf Appendix C: A Weak LLC for $\PGSp$}  \label{S:weakLLC}
We have mentioned that the ideas involved in the construction of the Spin lifting  in \S \ref{S:triality} via  an application of triality and the results of Kret-Shin allow one to construct a weak LLC for $\PGSp$ refining the results of B. Xu in this specific case. In this appendix, let us give an outline of this, explaining in particular what ``weak" means.  The reader will notice that the general structure is very similar to our treatment of the LLC for $G_2$.  
\vskip 5pt

\subsection{\bf Weak equivalence of L-parameters} 
In  \S \ref{SS:Spin-KS}, we have encountered the natural map
\[  {\rm std}_* \times \spin_*: \Phi(\PGSp) \longrightarrow \Phi(\Sp) \times \Phi(\GL_8). \]
As mentioned there, this map is not injective, because $\Spin_7(\C)$ is not acceptable in the sense of Larsen \cite{CG1}.  On the other hand, the standard and Spin L-functions of $\phi \in \Phi(\PGSp)$ is completely determined by its image in $\Phi(\Sp) \times \Phi(\GL_8)$. This motivates us to define the coarser equivalence relation on L-parameters:
\vskip 5pt

\begin{definition}
Two elements $\phi_1, \phi_2 \in \Phi(\PGSp)$ are said to be weakly equivalent if 
\[  ({\rm std}_*(\phi_1), \spin_*(\phi_1)) = ({\rm std}_*(\phi_2), \spin_*(\phi_2)) \in \Phi(\Sp) \times \Phi(\GL_8). \]
We denote the set of weak equivalence classes by $\Phi_w(\PGSp)$. 
\end{definition}
\vskip 5pt

One has \cite[\S 1]{CG1}:
\vskip 5pt

\begin{lemma}
The following are equivalent for two elements $\phi_1$ and  $\phi_2$ of $\Phi(\PGSp)$:
\vskip 5pt
\begin{itemize}
\item $\phi_1$ and $\phi_2$ are weakly equivalent;
\item $\phi_1$ and $\phi_2$ are elementwise conjugate.
\end{itemize}
\end{lemma}
\vskip 5pt
Our goal in this appendix is to construct  a canonical map 
\[  \mathcal{L}_w:  \Irr(\PGSp) \longrightarrow \Phi_w(\PGSp) \]
satisfying some desirable properties.
\vskip 5pt

\subsection{\bf Theta dichotomy}
We shall make use of the similitude theta correspondence associated to the following dual pairs:
\[
\xymatrix@R=2pt{
&&\PGO_8\\
&\PGSp \ar@{-}[ru]\ar@{-}[ld]\ar@{-}[rd]&\\
\PGO_{5,1}&&\PGO_6 
}
\]
Here, the orthogonal similitude groups in the right tower are split, whereas the one on the left has $F$-rank $1$ and is isomorphic to $\PGL_2(D) \rtimes \Z/2\Z$ where $D$ is the quaternion division $F$-algebra.  
\vskip 5pt

As a consequence of the conservation relation \cite{SZ},  any $\sigma \in {\rm Irr}(\PGSp)$ has a nonzero theta lift to exactly one of $\PGO_8$ or $\PGO_{5,1}$. This gives a partition
\[  \Irr(\PGSp)  = \Irr^{\heartsuit}(\PGSp)  \bigsqcup  \Irr^{\spadesuit}(\PGSp)  \]
where the elements in $\Irr^{\heartsuit}(\PGSp)$ are those with nonzero theta lifts to $\PGO_8$.  Among the representations in $\Irr^{\heartsuit}(\PGSp)$, some will participate in the theta correspondence with the lower step of the tower, i.e. with $\PGO_6 = \PGL_4(F) \rtimes \Z/2\Z$. Thus, we have a further decomposition
\[  \Irr^{\heartsuit}(\PGSp) = \Irr^{\diamondsuit}(\PGSp)  \sqcup  \Irr^{\clubsuit}(\PGSp)  \]
where $\Irr^{\clubsuit}(\PGSp)$ denotes the subset of those representations which has nonzero theta lifting to $\PGO_6$. It will be better to group $\Irr^{\spadesuit}(\PGSp)$ and $\Irr^{\clubsuit}(\PGSp)$ together, denoting their union by $\Irr^{\spadesuit\clubsuit}(\PGSp)$.
\vskip 5pt

In short, the similitude theta correspondence gives maps:
\[  \begin{CD}
\theta^{\spadesuit\clubsuit}:    \Irr^{\spadesuit\clubsuit}(\PGSp)  @>>>  \Irr(\PGL_2(D))_{/\Z/2\Z} \sqcup \Irr(\PGL_4)_{/ \Z/2\Z}  @>{\rm JL}>> \Irr(\PGL_4)_{/ \Z/2\Z}.
\end{CD} \]
and
\[  \begin{CD}
\theta^{\diamondsuit}:  \Irr^{\diamondsuit}(\PGSp)  @>>>  \Irr^{\heartsuit}(\PGSp)  @>\theta^{\heartsuit}>> \Irr(\PGO_8)  @>>> \Irr(\PGSO_8)_{\Z/2\Z}. \end{CD} \] 
The latter map is in fact injective and takes value in the $\Z/2\Z$-fixed points in $\Irr(\PGSO_8)$.
\vskip 10pt

\subsection{\bf Spin lifting}
We have defined the Spin lifting of generic representations of $\PGSp$ in \S \ref{SS:spin-lift}. In fact, this definition can be extended to the whole of $\Irr(\PGSp)$ to give:
\[   \spin_*:  \Irr(\PGSp)  \longrightarrow \Irr(\GL_8).\]
 We consider two cases, according to the decomposition of $\Irr(G)$. 
\vskip 5pt

\noindent (a) if $\sigma \in \Irr^{\heartsuit}(\PGSp)$, we set:
\[  f_2^*(\theta^{\heartsuit}(\pi)) \in \Irr(\SO_8)_{/\PGSO_8 } \]
and
\[ \spin_*(\sigma) =  \mathcal{A}\left(  f_2^*(\theta^{\heartsuit}(\pi)) \right) \in \Irr(\GL_8). \]
 
 Here, we refer to \S \ref{SS:triality} for the definition of $f_2$. This definition of $\spin_*$ is the same as that given in \S \ref{SS:spin-lift} for generic representations.
 \vskip 10pt

\noindent (b) if $\sigma \in \Irr^{\spadesuit \clubsuit}(\PGSp)$, then recall that we have the theta lifting map
\[   \theta^{\spadesuit\clubsuit}:    \Irr^{\spadesuit \clubsuit}(G)    \longrightarrow \Irr(\PGL_4)_{/ \Z/2\Z}.
\]
Moreover, one has a commutative diagram
\[  \begin{CD}  
 \Irr(\PGL_4)_{/ \Z/2\Z} @>>>  \Phi(\PGL_4)_{/ \Z/2\Z}  \\
 @V{\boxplus}VV  @VV{\oplus}V \\ 
\Irr(\GL_8)  @>\mathcal{L}_{\GL_8}>>  \Phi(\GL_8)  
\end{CD} \]
where the horizontal arrows are LLC maps, the first vertical arrow is the isobaric sum and the second vertical arrow is the map sending $(\phi, \phi^{\vee})$ to 
$\phi \oplus \phi^{\vee}$. More precisely, $\boxplus$ is given concretely as:
\[  \boxplus: (\tau, \tau^{\vee}) \mapsto \text{the unique constituent of $\Ind_P^{\GL_8} (\tau \boxtimes \tau^{\vee})$ with $L$-parameter $\phi_{\tau} \oplus \phi_{\tau}^{\vee}$}, \]
where $P$ is the maximal parabolic subgroup of $\GL_8$ with Levi factor $\GL_4 \times \GL_4$.
\vskip 5pt

 We may now set:
 \[  \spin_*(\sigma) =  \boxplus \circ \theta^{\spadesuit \clubsuit}(\sigma)  \in \Irr (\GL_8). \]
 The representations in $\Irr^{\clubsuit}(\PGSp)$ have been treated in both (a) and (b) above. One can check that the two definitions of $\spin_*$  agree on $\Irr^{\clubsuit}(\PGSp)$. Thus, we have now a well-defined map
\[  \begin{CD} 
\spin_*: \Irr(\PGSp) @>>>\Irr(\GL_8). \end{CD} \]

\vskip 5pt

\subsection{\bf Kret-Shin parameters} 
 Now a key result in the construction of our map $\mathcal{L}_w$ is the following:
 
\begin{thm}  \label{T:exist}
Given $\sigma \in \Irr(\PGSp)$, there exists a unique $\phi \in \Phi_w(\PGSp)$ such that
\[  
\std \circ \phi  = \mathcal{L}_{\Sp}({\rm rest}(\sigma))  \]
and
\[  \spin \circ \phi = \mathcal{L}_{\GL_8}(  \spin_*(\sigma)). \]
Here we recall that ${\rm rest}: \Irr(\PGSp) \longrightarrow \Irr(\Sp)_{/\PGSp}$ is the restriction of representations from $\PGSp$ to $\Sp$.
\end{thm}

Proposition \ref{P:KS} is a special case of this for $\sigma \in \Irr_{gen, ds}(\PGSp)$. Recall that we have called a $\phi$ as in the Theorem a Kret-Shin parameter of $\sigma$.

\vskip 5pt
\begin{proof}
To prove the Theorem, we consider two cases:
\vskip 5pt

\begin{itemize} 
\item If $\sigma\in \Irr^{\spadesuit \clubsuit}(\PGSp)$, then  our discussion in Case (b)  in the previous subsection already gives the desired $\phi$. In the notation there, we set
 \[  \phi = j \circ \phi_{\theta^{\spadesuit\clubsuit}(\sigma)}, \]
 where 
 \[  j:   \PGL_4^{\vee} = \Spin_6(\C) \longrightarrow \Spin_7(\C). \]
 This satisfies $\spin \circ \phi =  \mathcal{L}_{\GL_8}(  \spin_*(\sigma))$. Moreover, by the theory of theta correspondence, it is easy to see that 
 \[  \std \circ \phi = \mathcal{L}_{\SO_6}(\theta^{\spadesuit\clubsuit}(\sigma)) \oplus 1  =     \mathcal{L}_{\Sp_6}({\rm rest}(\sigma)).  \] 
\vskip 5pt

\item Suppose now that $\sigma \in \Irr^{\heartsuit}(\PGSp)$.  If $\sigma$ is not a discrete series representation, then  one can determine its local theta lift to $\PGSO_8$ completely (using the fact that the LLC for Levi subgroups of $\PGSp$ and $\PGSO_8$ are known).  From this knowledge, one can write down the desired $\phi \in \Phi(\PGSp)$: it will factor through a Levi subgroup of $\Spin_7(\C)$. We omit the details here.
\vskip 5pt

The key case is that of discrete series representations. Here the proof is along the same lines  as that of Proposition \ref{P:KS}, where we demonstrated this for generic discrete series representations. To carry out the same proof, for $\sigma \in \Irr_{ds}(\PGSp)$, one needs to globalize $\sigma$ appropriately. We discuss this globalization below.
\end{itemize}
 
\vskip 5pt

\subsection{\bf Globalization}

 Let $k$ be a totally real number field with adele ring $\mathbb{A}$ and a place $v_0$ such that $k_{v_0} = F$, and let $v_1\ne v_0$ be another finite place. Consider  the split group $\PGSp$ over $k$.  We have:
 \vskip 5pt
 
 \begin{lemma}  \label{L:global}
 There is a cuspidal automorphic representation $\Sigma$ of $\PGSp$ satisfying:
 \vskip 5pt
 \begin{itemize}
 \item $\Sigma_{v_0} = \sigma$;
 \item $\Sigma_{v_1} = {\rm St}_{v_1}$ is the Steinberg representation;
 \item for each real place $v_{\infty}$ of $k$, $\Sigma_{v_{\infty}}$ is a generic discrete series representation with very regular infinitesimal character so that it is spin-regular in the sense of \cite{KS}.;
 \item $\Sigma_v$ is a quadratic twist of an unramified representation for all other $v$.
 \end{itemize}
  \end{lemma}

 \begin{proof}
 Using \cite[Theorem 5.8]{S}, we can find a cuspidal automorphic representation $\Sigma'$ of $\PGSp$ satisfying:
 \vskip 5pt
 
 \begin{itemize}
 \item $\Sigma'_{v_0} = \pi$;
 \vskip 5pt
 
 \item $\Sigma'_{v_1} = {\rm St}_{v_1}$;
 \vskip 5pt
 
 \item for each real place $v_{\infty}$, $\Sigma'_{v_{\infty}}$ is a discrete series representation with very regular infinitesimal character and is spin-regular.
 
 \vskip 5pt
 
 \item for all other places $v$, $\Sigma'_v$ is unramified.
  \end{itemize}
 The only condition potentially missing from $\Sigma'$ is the genericity of its archimedean components.
 \vskip 5pt
 
 Now the restriction of $\Sigma'$ to $\Sp$ determines a near equivalence class of square-integrable automorphic representations with a generic  A-parameter $\Psi$ of $\Sp$, because of the Steinberg local component.   Moreover, again because of the Steinberg local component, the global component group of the A-parameter $\Psi$ is trivial.  
  As a consequence, every representation of $\Sp(\mathbb{A})$ in the global A-packet attached to $\Psi$ occurs in the space of cusp forms by Arthur's multiplicity formula \cite{A}. In particular,   if we take an irreducible automorphic summand $\Sigma'^{\flat}$ of $\Sigma'|_{\Sp}$, we 
  may replace the archimedean components of  $\Sigma'^{'\flat}$ by the generic discrete series representations in the local A-packets (which are just L-packets) and obtain a cuspidal automorphic representation $\Sigma^{\flat}$ in the same A-packet such  that $\Sigma^{\flat}_f = \Sigma'^{\flat}_f$.  Note also that the central character of $\Sigma^{\flat}$ is the same as that of $\Sigma'^{\flat}$, which is trivial.
 \vskip 5pt
 
 Now let $\Sigma$ be a cuspidal automorphic representation of $\PGSp$ whose restriction to $\Sp$ contains $\Sigma^{\flat}$. Then the archimedean components of $\Sigma$ are generic discrete series which are spin-regular. Moreover, $\Sigma_{v_0} = \pi \otimes \chi_0$ and $\Sigma_{v_1} = {\rm St}_{v_1} \otimes \chi_1$ for some quadratic characters $\chi_0$ and $\chi_1$.    Thus, after replacing $\Sigma$ by its twist  by an automorphic quadratic character $\chi$ such that $\chi_{v_0} = \chi_0$ and $\chi_{v_1} = \chi_1$, we obtain the desired $\Sigma$ as in the Lemma.   
 \end{proof}
 \vskip 10pt
 
 \subsection{\bf Global Theta lift.}
 With the cuspidal automorphic representation $\Sigma$ given by the lemma, we may now 
  consider the global theta lift of $\Sigma$ to the split group $\PGSO_8$ over $k$. We have:
  \vskip 5pt
  
  \begin{lemma}  \label{L:theta}
  The global theta lift $\Theta(\Sigma)$ of $\Sigma$ to $\PGSO_8$ is a nonzero cuspidal automorphic representation.
  \end{lemma}
  \vskip 5pt
  
 \begin{proof}
 We first note that $\Theta(\Sigma)$ must be cuspidal because of the tower property of theta correspondence.
 Indeed, the Steinberg representation  $St_{v_1}$ of $\PGSp(k_{v_1})$ does not participate in the local theta correspondence with $\PGSO_6(k_{v_1})$, so that the global theta lift of $\Sigma$ to the split $\PGO_6$ is zero.  \vskip 5pt

  By \cite[Thm. 1.4(ii)]{GQT}, to show the nonvanishing of the global theta lift $\Theta(\Sigma)$, it suffices to show:
  \vskip 5pt
  
  \begin{itemize}
  \item For any place $v$ of $k$, the local component $\Sigma_v$ has nonzero theta lift to $\PGSO_8(k_v)$. This follows because unramified representations or generic representations have nonzero theta lifting to $\PGSO_8(k_v)$; this takes care of all places $v \ne v_0$, and for the place $v_0$, it follows by our assumption that $\Sigma_{v_0} = \sigma \in \Irr^{\heartsuit}(\PGSp)$. 
  \vskip 5pt
  
  \item The degree 7 standard L-function $L(s, \Sigma, std)$ is nonzero at $s=1$.  This follows because $\Sigma|_{\Sp}$  has a generic A-parameter.
  \end{itemize}
  
 By the Rallis inner product formula \cite[Thm. 1.4(ii)]{GQT}, we deduce that the global theta lift $\Theta(\Sigma)$ to $\PGSO_8$ is nonzero and the lemma is proved.
 \end{proof}
 
 \vskip 5pt
 
 We may thus consider the global Spin lifting of $\Sigma$:
 \[ \spin_*(\Sigma) =  \mathcal{A}(f_2^*(\Theta(\Sigma)). \]
 Then $\spin_*(\Sigma)$ is  in fact a cuspidal representation of $\GL_8$ (because of the spin-regularity of the infinitesimal character at infinite places).  Applying the results of Kret-Shin, as recounted in \S \ref{SS:KS}, and arguing as in the proof of Proposition \ref{P:KS}, we conclude the proof of  Theorem \ref{T:exist}.
 \end{proof}
 
 \vskip 10pt
 
 \subsection{\bf  Construction of $\mathcal{L}_w$} 
 By Theorem \ref{T:exist}, we have shown:
 \vskip 5pt
 
 \begin{thm} \label{T:main-appendix}
 Let 
 \[  \mathcal{L}_w: \Irr(\PGSp) \longrightarrow \Phi_w(\PGSp) \]
 be defined by
 \[  \mathcal{L}_w (\sigma) := \text{the unique  $\phi$ associated to $\sigma$ by Theorem \ref{T:exist}} \]
 for each $\sigma \in \Irr(\PGSp)$. Then $\mathcal{L}_w$ features in and is characterized by the following commutative diagram:
 \[  \begin{CD}
 \Irr(\PGSp) @>\mathcal{L}_w>>  \Phi_w(\PGSp) \\
 @V{\rm rest} \times \spin_*VV  @VV{\rm std}_* \times \spin_*V \\
 \Irr(\Sp) \times \Irr(\GL_8) @>>\mathcal{L}_{\Sp} \times \mathcal{L}_{\GL_8}> \Phi(\Sp) \times \Phi(\GL_8) 
 \end{CD} \]
 \vskip 5pt
 \noindent Moreover,  the map $\mathcal{L}_w$ sends tempered (respectively discrete series) representations to (weak equivalence classes of ) tempered (respectively discrete series) L-parameters and satisfies 
 \[  \mathcal{L}_w(\sigma \otimes \chi) \cong \mathcal{L}_w(\sigma) \otimes \chi \]
  for any quadratic character $\chi$ of $F^{\times}$. 
  \end{thm}
\vskip 5pt

\begin{proof}
It remains to verify the last statement: the compatibility of $\mathcal{L}_w$ with twisting by quadratic characters, 
It comes down to the fact that the similitude theta correspondence 
is compatible with such twisting. More precisely, if $\sigma \in \Irr^{\heartsuit}(\PGSp)$, then 
\[  \theta^{\heartsuit}(\sigma\otimes \chi) = \theta^{\heartsuit}(\sigma) \otimes \chi . \]
Here, it is important to note that the similitude character   of $\PGSO_8$ implicit on the right hand side of the equation is the one which is trivial on $f_1(\SO_8(F))$. In particular, it is not trivial on $f_2(\SO_8(F))$ (see \S \ref{SS:triality} for $f_1$ and $f_2$).  
Thus, for any representation $\pi$ of $\PGSO_8(F)$,
\[  f_2^*(\pi \otimes \chi )) = f_2^*(\pi) \otimes \chi  \]
where we regard $\chi$ as a character of $\SO_8(F)$ via composition with the spinor norm. 
Since the LLC for $\SO_8$ is compatible with twisting by quadratic characters, we deduce that
\[  \spin_*(\sigma \otimes \chi )  =  \spin_*(\sigma) \otimes  (\chi \circ \det) \in \Irr (\GL_8)  \]
For $\sigma \in \Irr^{\spadesuit\clubsuit}(G)$, one verifies this identity along similar lines; we omit the details. 
\end{proof}

\vskip 10pt

\subsection{\bf Fibers of $\mathcal{L}_w$} 
Finally, we would like to understand the fibers of $\mathcal{L}_w$. We shall show:
\vskip 5pt

\begin{thm}
(i) The map $\mathcal{L}_w$  is surjective. Moreover, the natural restriction map
 \[  {\rm rest}: \mathcal{L}_w^{-1}(\phi)  \longrightarrow \mathcal{L}^{-1}_{\Sp} (\std \circ \phi)_{/\PGSp} \]
 is surjective.

\vskip 5pt

(ii) The fibers of $\mathcal{L}_w$ are unions of Xu's packets (introduced in \S \ref{S:Xu}).
\end{thm}

\vskip 5pt

\begin{proof}
(i)  For given $\phi \in \Phi_w(\PGSp)$, with $\phi^{\flat} := \std \circ \phi \in \Phi(\Sp)$, write 
\[  \Pi_{\phi^{\flat}} :=  \mathcal{L}^{-1}_{\Sp} (\std \circ \phi) \]
  be the corresponding L-packet for $\Sp$, which is  non-empty. Pick $\sigma^{\flat} \in \Pi_{\phi^{\flat}}$. Because $\phi^{\flat}$ can be lifted to $\Spin_7(\C)$, the representations in $\Pi_{\phi^{\flat}}$ have trivial central character. Hence we can pick 
 $\sigma \in \Irr(\PGSp)$ such that $\sigma^{\flat} \subset {\rm rest}(\sigma) $. Then $\mathcal{L}_w(\sigma)$ is a lift of $\phi^{\flat}$, so that  
 \[  \mathcal{L}_w(\sigma) = \phi \otimes \chi \]
  for some quadratic character $\chi$. Hence $\mathcal{L}(\sigma \otimes \chi ) = \phi$. This proves the first assertion.
Indeed, the proof shows that for any $\sigma^{\flat} \in \Pi_{\phi^{\flat}}$, there exists $\sigma \in \Irr(\PGSp)$ such that $\sigma^{\flat} \subset {\rm rest}(\sigma)$ and 
$\mathcal{L}_w(\sigma) = \phi$.  This gives the second assertion.
 \vskip 5pt
 
 (ii) It is not hard to reduce this issue to the case of a discrete series L-parameter $\phi$.
 Moreover, we shall assume that $\phi$ does not factor through $\Spin_6(\C)$ so that 
 \[  \mathcal{L}_w^{-1}(\phi)  \subset \Irr^{\diamondsuit}(\PGSp). \]
 The case when $\mathcal{L}_w^{-1}(\phi) \subset \Irr^{\spadesuit\clubsuit}(\PGSp)$ is checked along similar lines and we leave it to the reader.
\vskip 5pt

Suppose then that $\sigma \in \mathcal{L}_w^{-1}(\phi)$. Writing $\phi^{\flat} = {\rm std} \circ \phi$, we know that $\sigma$ lies in a unique Xu's packet $\tilde{\Pi}^X_{\phi^{\flat}}$, where we are using the notations from \S \ref{S:Xu}. Thus, we need to show that
\[  
\tilde{\Pi}^X_{\phi^{\flat}} \subset \mathcal{L}_w^{-1}(\phi). \]
By Lemma \ref{L:global}, we globalize $\sigma$ to a cuspidal automorphic representation $\Sigma$ over a number field $k$ such that $\Sigma_{v_0} = \sigma$ and $\Sigma_{v_1} = {\rm St}_{v_1}$ is the Steinberg representation; moreover, the global theta lift $\Theta(\Sigma)$ of $\Pi$ to $\PGSO_8$ is a nonzero cuspidal representation by Lemma \ref{L:theta}. 
\vskip 5pt

As we saw in the proof of Lemma \ref{L:global}, ${\rm rest}(\Sigma)$ gives rise to a generic A-parameter $\Psi^{\flat}$ with trivial global component group $S_{\Psi^{\flat}}$.
 Now we appeal to the global results of Xu \cite[Thm. 4.1]{Xu3}. Under these circumstances, Xu showed that $\Sigma$ belongs to a global generic A-packet for $\PGSp$ of the form
\[   \tilde{\Pi}^X_{\Psi^{\flat}}  = \otimes_v \tilde{\Pi}_{\Psi^{\flat}_v}^X  \]
where the local packets $\tilde{\Pi}_{\Psi^{\flat}_{v}}^X$ were introduced in \S \ref{S:Xu}. The contribution of 
$ \tilde{\Pi}^X_{\Psi^{\flat}} $  to the automorphic discrete spectrum is governed by the Arthur multiplicity formula. In particular, since the global component group $S_{\Psi^{\flat}}$ is trivial, \cite[Thm. 4.1]{Xu3} says that every element of $\tilde{\Pi}^X_{\Psi^{\flat}}$ is automorphic.     
\vskip 5pt

Now we know that $\sigma = \Sigma_{v_0} \in \Pi_{\Psi^{\flat}_{v_0}}^X$  and we would like to show that
\[   \tilde{\Pi}_{\Psi^{\flat}_{v_0}}^X  \subset \mathcal{L}_w^{-1}(\phi).  \]
 For any $\sigma'  \in \tilde{\Pi}_{\Psi^{\flat}_{v_0}}^X$, we thus need to show that 
\[  \mathcal{L}_w(\sigma') = \mathcal{L}_w(\sigma) = \phi. \]
 For this, let $\Sigma'$ be the representation of $PGSp_6(\mathbb{A})$ such that
\[  \Sigma'_{v_0} = \sigma' \quad \text{and} \quad \Sigma'_v = \Sigma_v \, \, \text{for all $v \ne v_0$.}\]
As we noted above,  $\Sigma'$ is cuspidal automorphic as well. Moreover, as we showed in the proof of Lemma \ref{L:theta}, the global theta lift $\Theta(\Sigma')$ to $\PGSO_8$ is nonzero.  
\vskip 5pt

Now we have two cuspidal representations $\Theta(\Sigma)$ and $\Theta(\Sigma')$ on $\PGSO_8$ which are nearly equivalent. Hence
\[  \mathcal{A} (f_2^*(\Theta(\Sigma))) \,\, \text{ and } \, \,  \mathcal{A} (f_2^*(\Theta(\Sigma')))  \]
give the same generic A-parameter on $\GL_8$. In particular, their local component at $v_0$ are the same.
Likewise,  both $f_1^*(\Theta(\Sigma))$ and $f_1^*(\Theta(\Sigma'))$ have A-parameter  $1 \oplus \Psi^{\flat}$.
Hence, on extracting the $v_0$-component,     it follows that $\mathcal{L}_w(\sigma') = \mathcal{L}_w(\sigma) = \phi$, as desired.
 \end{proof}
\vskip 5pt

As one believes that Xu's packets are precsiely the L-packets of $\PGSp$, the map $\mathcal{L}_w$ is not the ultimate LLC map. However,
it is a strict refinement of the results of Xu in the setting of $\PGSp$. Indeed, as we showed in the main body of the paper (see Proposition \ref{P:KS2}), 
if $\phi \in \Phi(\PGSp)$  is valued in $G_2(\C)$, then $\mathcal{L}_w^{-1}(\phi)$ is a single Xu's packet and hence is an honest L-packet of $\PGSp$. This fact plays a crucial role in our proof of the LLC for $G_2$.  
\vskip 10pt

  \vskip 5pt
 
 \noindent{\bf Acknowledgments:} 
The final part of this work was completed when the authors participated in a Research in Teams (RIT) project ``Modular Forms and Theta Correspondence for Exceptional Groups" hosted by the Erwin Schrodinger Institute, University of Vienna. We thank the Erwin Schrodinger Institute for its generous support and wonderful working environment, as well as our other RIT members Nadya Gurevich and Aaron Pollack for helpful discussions.  We would like to thank B. Bekka and M. Tadi\'c for their help with the Fell topology and G. Chenevier for his help with the unacceptability of $\Spin_7$. Finally, we are grateful to Benedict Gross, whose inspirational ideas on the exceptional theta correspondence have been sustaining our work for the past 25 years.
W.T. Gan is partially supported by a Singapore government MOE Tier 1 grant R-146-000-320-114.  G. Savin is partially supported by 
 a   National Science Foundation grant DMS-1901745.

\end{document}